\documentclass{amsart}
\usepackage{graphicx}
\newtheorem{thm}{Theorem}[section]
\newtheorem{cor}[thm]{Corollary}
\newtheorem{lem}[thm]{Lemma}
\newtheorem{prop}[thm]{Proposition}
\theoremstyle{definition}
\newtheorem{defn}[thm]{Definition}
\theoremstyle{remark}
\newtheorem{rem}[thm]{Remark}
\numberwithin{equation}{section}

\begin{document}

\title[Characterizing stationary probabilities in DPS systems]
{Characterization and asymptotic analysis\\ of the stationary probabilities\\
in Discriminatory Processor Sharing Systems}%
\author{Vyacheslav M. Abramov}%
\address{Swinburne University of Technology, Faculty of
Engineering and Industrial Systems, Hawthorn Campus, John Street, PO Box 218\\
Hawthorn, Victoria 3122, Australia}%
\email{vabramov126@gmail.com}%


\begin{abstract}
In this paper, we establish two different results. The first result
is a characterization theorem saying that if the stationary state
probabilities for originally described Markovian discriminatory
processor sharing (DPS) system have a closed product geometric form
(the exact definition is given in the paper), then the system must
only be Egalitarian, i.e. all flows in this system must have equal
priorities. The second result is the tail asymptotics for the
stationary probabilities. We provide a detailed asymptotic analysis
of the system, and obtain the exact asymptotic form of the
stationary probabilities in DPS systems when the number of flows in
the system is large.
\end{abstract}
\maketitle
\section{Introduction} 
Discriminatory processor sharing (DPS) policy was originally
introduced and studied by Kleinrock \cite{Kleinrock1} under the
title priority processor sharing. It is an extension of usual
(non-priority) processor sharing (PS) policy, which also was
originally introduced by Kleinrock \cite{Kleinrock2}. The DPS system
is defined as follows. Suppose that there are $I$ flow classes. All
flows are served independently of each other. They share the service
time as follows. If there are $n_1$, $n_2$,\ldots, $n_I$ flows in
the system of the classes 1,2, \ldots, $I$, respectively, then the
rate of shared service of a class $i$ flow is
$$
\frac{g_i}{\sum_{l=1}^I g_ln_l},
$$
where $g_1$, $g_2$,\ldots, $g_I$ are `weights' of flows belonging to
the corresponding classes. Although the DPS policy was introduced
long time ago, the progress in its investigation is very limited.
The first substantial contribution to the theory of DPS systems was
due to Fayolle, Mitrani and Iasnogorodski \cite{FMIa}. These authors
derived the system of integro-differential equations for the
conditional expectation of the response time of a flow (the time
spent in the system by a flow of a given class arriving in the
system) given that the required service time of the flow exceeds the
level $t$ for the $M/G/1$ DPS system with $I$ flow classes, and
provided a detailed study of that system of equations. Additional
study of the system of integro-differential equations \cite{FMIa} is
given in Avrachenkov et al. \cite{Avrachenkov et al}. The stationary
queue-length distributions and heavy-traffic behavior for Markovian
DPS system have been studied by Rege and Sengupta \cite{Rege and
Sengupta}. The similar analysis for the phase-type service DPS
system has been provided by Verloop, Ayesta and N\'{u}\~{n}ez-Queija
\cite{Verloop et al}, who also established state-space collapse
property for the heavy-traffic behavior. Bonald and Prouti\`{e}re
\cite{BonaldProutiere1}, \cite{BonaldProutiere2} and
\cite{BonaldProutiere3} provided intensive study of a certain class
of PS systems. They classified those systems and studied their
important properties such as insensitivity and balance properties as
well as established certain bounds for so-called monotonic PS
networks that include DPS systems as a particular case. For other
known results in the area of DPS systems see also the review papers
by Altman, Avrachenkov and Ayesta \cite{Altman et al} and Aalto et
al. \cite{Aalto et al}, and for recent results related to large
deviation of monotonic PS networks that include DPS system see
\cite{JonckheereLopez}.

The present paper contains two important results. The first result
is a simple characterization theorem telling us about the
possibility to represent their stationary probabilities in closed
form. Characterization of queueing system is an established area in
queueing theory. Most of the results of this theory are associated
with inverse problems (see e.g. the book by  Kalashnikov and Rachev
\cite{Kalashnikov}; see also \cite{Abramov} for one of recent
results).

The second one is an asymptotic theorem on the tail behavior of the
stationary probabilities, when the numbers flows in the system is
large. For review of the different approaches the light tail
asymptotics see \cite{Miyazawa}. For a recent study of tail
asymptotics in PS queueing systems see \cite{Robert} and that in
priority queueing systems see \cite{Li}.

 Up
to this time, the important properties of the stationary
distributions of DPS systems have been studied with the aid of the
vector-valued $z$-transforms having a complicated form \cite{Rege
and Sengupta}, \cite{Verloop et al}. Such an approach is
straightforward, and it makes the analysis of the system
characteristics cumbersome. Unlike many papers in this area
(including aforementioned ones \cite{Rege and Sengupta},
\cite{Verloop et al}), the present paper does not use the
traditional $z$-transform method. It is based on a direct study of
the system of equations for this system.

Our approach uses the same bounds as those in Bonald and
Prouti\`{e}re \cite{BonaldProutiere3}. We prove that these bounds
asymptotically dominate the stationary probabilities in the DPS
system. Then, on the basis of these bounds we obtain the tail
asymptotics for the stationary probabilities of the DPS system.

Throughout the paper, empty sums are assumed to be set to zero and
empty products to one.

The rest of the paper is organized as follows. In Section
\ref{notation}, we describe the system, introduce notation and
formulate the results of the paper. In Section \ref{Proofs}, we
introduce necessary concepts and prove the main results of the
paper. In Section \ref{Num}, we define the most likely direction of
the process when the number of flows in the system is large and
provide its numerical study. In Section \ref{Conclusion}, we
conclude the paper and formulate an open problem.


\section{Description of the system, notation and main results}\label{notation}

Consider single server queueing system with $I$ classes of flows.
Flows of the $i$th class ($i$-flows) arrive in the system according
to an ordinary Poisson process with rate $\lambda_i$. The nominated
service time distribution of an $i$-flow is exponential with
parameter $\mu_i$. Denote the load parameter of $i$-flows by
$\rho_i=\frac{\lambda_i}{\mu_i}$, and assume
$\rho=\rho_1+\rho_2+\ldots+\rho_I<1$. All flows presented in the
system are served simultaneously, and share the service according to
the DPS policy with the vector $\mathbf{g}=(g_1,g_2,\ldots,g_I)$.
The word \textit{nominated} means that each single $i$-flow in the
system, that does not share its service, is being served
exponentially with parameter $\mu_i$ unless new arrival in the
system does not occur, and occasionally its service can be finished
before a new arrival. The assumption $\rho<1$ means that the system
is stable.

Let $\mathbf{Q}(t)=(Q_1(t),Q_2(t),\ldots,Q_I(t))$ denote the
vector-valued queue-length process at time $t$, where $Q_i(t)$
denotes the number of $i$-flows in the system are being served at
time $t$, and let
$P_{\mathbf{n}}=\lim_{t\to\infty}\mathsf{P}\{\mathbf{Q}(t)=\mathbf{n}\}$,
where $\mathbf{n}=(n_1,n_2,\ldots,n_I)$ is an integer-valued vector.
For a stable system, the last limit exists.

Throughout the paper we also use the following notation:
\begin{eqnarray*}
\mathbf{0}&=&(0,0,\ldots,0) - I\text{-dimensional vector of zeros},\\
\mathbf{1}&=&(1,1,\ldots,1) - I\text{-dimensional vector of ones},\\
\mathbf{1}_i&=&(\underbrace{0,\ldots,0}_{i-1\quad \text{zeroes}},1,
\underbrace{0,\ldots,0}_{I-i\quad \text{zeros}}),\\
\big<\mathbf{n},\mathbf{g}\big>&=&n_1g_1+n_2g_2+\ldots+n_Ig_I,\\
|\mathbf{n}|&=&\big<\mathbf{n},\mathbf{1}\big> \ = \
n_1+n_2+\ldots+n_I,\\
|\mathbf{n}|_i&=&n_1+n_2+\ldots+n_i; \quad
(|\mathbf{n}|_I=|\mathbf{n}|, |\mathbf{n}|_0=0).
\end{eqnarray*}
The inequality between the vectors is understood as the
componentwise inequalities. For example, $\mathbf{n}\geq\mathbf{0}$
means that all components of a vector $\mathbf{n}$ are nonnegative;
$\mathbf{n}>\mathbf{0}$ means that in a nonnegative vector
$\mathbf{n}$ there is at least one strictly positive component. A
vector $\mathbf{n}$ is said to be \textit{separated from zero} if
$n_i>0$ for all $i=1,2,\ldots,I$. The set of all vectors that are
separated from zero is denoted by $\mathcal{N}$.

Let ${\bf \Gamma}=(\gamma_1,\gamma_2,\ldots,\gamma_I)$ be a vector
of positive real numbers (vector of the direction). A vector
${\bf\Gamma}$ is called \textit{normalized} if
$\gamma_1+\gamma_2+\ldots+\gamma_I=1$. In the sequel, all vectors
${\bf\Gamma}$ considered in the paper are assumed to be normalized.

For $N=0,1,\ldots$ and a given vector of the direction
${\bf\Gamma}$, the set of the vectors $(\lfloor N\gamma_1\rfloor,
\lfloor N\gamma_2\rfloor,\ldots, \lfloor N\gamma_I\rfloor)$, where
for any real $a$, the symbol $\lfloor a\rfloor$ denotes the integer
part of $a$, is denoted $\mathcal{N}_{\bf\Gamma}$. Let $\mathcal{G}$
be an infinite set of directions ${\bf\Gamma}$ containing an
interior. We define the cone $
\mathcal{C}(\mathcal{G})=\cup_{{\bf\Gamma}\in\mathcal{G}}\mathcal{N}_{\bf\Gamma}
$.

For a positive integer $n$, denote
$\mathcal{N}_{{\bf\Gamma},n}=\left\{\right.\mathbf{n}\in\mathcal{N}_{\bf\Gamma}:
\mathbf{n}\geq (\lfloor n\gamma_1\rfloor,$ $\lfloor
n\gamma_2\rfloor,\ldots,$ $\lfloor n\gamma_I\rfloor)\left.\right\}$.

For a direction $\bf\Gamma$, let $n_{\bf\Gamma}$ be an indexed
integer number. Denote $\mathcal{N}(\mathcal{G})=\{n_{\bf\Gamma}:
{\bf\Gamma}\in\mathcal{G}\}$, and define the set
$\mathcal{C}(\mathcal{G},\mathcal{N}(\mathcal{G}))
=\cup_{{\bf\Gamma}\in\mathcal{G}}\mathcal{N}_{{\bf\Gamma},n_{\bf\Gamma}}$.

A vector
$\mathbf{n}\in\mathcal{C}(\mathcal{G},\mathcal{N}(\mathcal{G}))$ is
called \textit{boundary vector} of
$\mathcal{C}(\mathcal{G},\mathcal{N}(\mathcal{G}))$, if there exists
integer $i$, $i=1,2,\ldots,I$, such that $\mathbf{n}-\mathbf{1}_i$
does not belong to the set
$\mathcal{C}(\mathcal{G},\mathcal{N}(\mathcal{G}))$. The set of all
pairs $\{\mathbf{n},i\}$ where $\mathbf{n}$ is a boundary vector of
$\mathcal{C}(\mathcal{G},\mathcal{N}(\mathcal{G}))$ and
$\mathbf{n}-\mathbf{1}_i$ does not belong to the set
$\mathcal{C}(\mathcal{G},\mathcal{N}(\mathcal{G}))$ is denoted by
$\mathcal{C}^0(\mathcal{G},\mathcal{N}(\mathcal{G}))$.

In addition, for any integer $N$ and vector of direction
${\bf\Gamma}$, the following notation $\lfloor
N{\bf\Gamma}\rfloor=(\lfloor N\gamma_1\rfloor, \lfloor
N\gamma_2\rfloor,\ldots, \lfloor N\gamma_I\rfloor)$ is used.

\begin{defn}
The stationary probabilities of the vector valued queueing process
$\mathbf{Q}(t)$ are said to be presented in \textit{closed product
geometric form} if
\begin{eqnarray}
P_{\mathbf{n}}&=&(1-\rho)F(\mathbf{n},\mathbf{g})\prod_{i=1}^I\rho_i^{n_i},
\
\mathbf{n}> \mathbf{0},\label{eq-d1}\\
P_{\mathbf{0}}&=&1-\rho,\label{eq-d2}
\end{eqnarray}
for some function $F(\mathbf{n},\mathbf{g})$ depending only on the
vectors $\mathbf{n}$ and $\mathbf{g}$ (and hence independent of the
vector-valued parameters $(\lambda_1$, $\lambda_2$, \ldots,
$\lambda_I$) and $(\mu_1$, $\mu_2$, \ldots, $\mu_I$)).
\end{defn}

\begin{thm}\label{thm-3}
The stationary probabilities $P_{\mathbf{n}}$ of the DPS queueing
system can be represented in closed product geometric form if and
only if the components of the vector $\mathbf{g}$ all are equal,
that is, in the only case of Egalitarian PS system.
\end{thm}

For the formulation of the next main theorem, we introduce the
following notation. For $i=1,2,\ldots,I-1$ and positive real numbers
$\gamma_1$, $\gamma_2$, \ldots, $\gamma_I$
($\sum_{i=1}^I\gamma_i=1$) set
\begin{equation}\label{eq-d11}
\begin{aligned}
\theta^{(1)}_i=&
\exp\left\{\sum_{j=1}^{I-i}\left(\frac{g_i}{g_{i+j}}-1\right)
\ln\left(1+\frac{\gamma_{i+j}g_{i+j}}{\sum_{k=1}^{i+j-1}\gamma_kg_k}\right)\right\}
\end{aligned}
\end{equation}
and for $i=I$ set $\theta^{(1)}_I=1$. Similarly, for $i=1$ set
$\theta^{(2)}_1=1$ and for $i=2,3,\ldots,I$ and the same positive
real numbers $\gamma_1$, $\gamma_2$, \ldots, $\gamma_I$ set
\begin{equation}\label{eq-d12}
\begin{aligned}
\theta^{(2)}_i=&\exp\left\{\sum_{j=1}^{i-1}\left(\frac{g_i}{g_{j}}-1\right)
\ln\left(1+\frac{\gamma_{j}g_{j}}
{\sum_{k=j+1}^I\gamma_kg_k}\right)\right\}.
\end{aligned}
\end{equation}

For $i=1,2,\ldots,I$, introduce the following values:
\begin{eqnarray}
\Delta_i^{(1)}&=&\frac{\gamma_1g_1+\gamma_2g_2+\ldots+\gamma_Ig_I}{\gamma_ig_i}
\cdot{\theta_i^{(1)}}\cdot{\rho_i},\label{eq-d14}\\
\Delta_i^{(2)}&=&\frac{\gamma_1g_1+\gamma_2g_2+\ldots+\gamma_Ig_I}{\gamma_ig_i}
\cdot{\theta_i^{(2)}}\cdot{\rho_i}.\label{eq-d15}
\end{eqnarray}

For integer parameter $N$, we set $n_1=\lfloor N\gamma_1\rfloor$,
$n_2=\lfloor N\gamma_2\rfloor$, \ldots, $n_I=\lfloor
N\gamma_I\rfloor$. Then $\mathbf{n}(N)=\lfloor
N{\bf\Gamma}\rfloor=\{\lfloor N\gamma_1\rfloor, \lfloor
N\gamma_2\rfloor, \ldots, \lfloor N\gamma_I\rfloor\}$.

The following theorem describes the asymptotic behavior of the
stationary probabilities.
\begin{thm}\label{thm-4}
Assume that $g_1<g_2<\ldots<g_I$, and
\begin{equation}\label{eq-l7}
\Delta_i^{(2)}<1, \quad i=1,2,\ldots,I.
\end{equation}
Then, as $N\to\infty$,
\begin{equation*}\label{eq-l9}
\begin{aligned}
\lim_{N\to\infty}\frac{P_{\lfloor
N{\bf\Gamma}\rfloor+\mathbf{1}_i}}{P_{\lfloor
N{\bf\Gamma}\rfloor}}=\Delta_i=\frac
{\Delta_i^{(1)}-c\Delta_i^{(2)}}{1-c},
\end{aligned}
\end{equation*}
$$
c=\frac{\sum_{i=1}^I\left[\mu_i{\gamma_ig_i}{\big(\sum_{l=1}^I
\gamma_lg_l\big)^{-1}}(\Delta_i^{(1)}-1)
-\lambda_i\big(1-\big(\Delta_i^{(1)}\big)^{-1}\big)\right]}{\sum_{i=1}^I
\left[\mu_i{\gamma_ig_i}
{\big(\sum_{l=1}^I\gamma_lg_l\big)^{-1}}(\Delta_i^{(2)}-1)
-\lambda_i\big(1-\big(\Delta_i^{(2)}\big)^{-1}\big)\right]}.
$$
\end{thm}

\begin{cor}\label{cor-1}
Under the assumptions of Theorem \ref{thm-4}
$$
\lim_{N\to\infty}\frac{\ln P_{\lfloor
N{\bf\Gamma}\rfloor}}{N}=\sum_{i=1}^I\gamma_i\ln\Delta_i.
$$
\end{cor}

\section{Proofs of the main results}\label{Proofs}

\subsection{Preliminaries}

In the following, the fractions in which both the numerator and
denominator are equal to zero, are set to zero. Specifically, the
fractions $\frac{n_i}{\big<\mathbf{n},\mathbf{g}\big>}$ in which
$\mathbf{n}=\mathbf{0}$ are set to zero. The system of linear
equations for the stationary probabilities $P_\mathbf{n}$,
$\mathbf{n}\geq\mathbf{0}$, which follows from the system of the
Chapman-Kolmogorov equations, is
\begin{equation}\label{eq-c2}
\sum_{i=1}^I\left[\frac{\mu_ig_i(n_i+1)}{\big<\mathbf{n}+\mathbf{1}_i,
\mathbf{g}\big>}P_{\mathbf{n}+\mathbf{1}_i}-\left(\lambda_i+
\frac{\mu_ig_in_i}{\big<\mathbf{n},
\mathbf{g}\big>}\right)P_{\mathbf{n}}+\lambda_iP_{\mathbf{n}-\mathbf{1}_i}\right]=0
\end{equation}
(see \cite{Rege and Sengupta}), where
$P_{\mathbf{n}-\mathbf{1}_i}=0$ in the case where the vector
$\mathbf{n}-\mathbf{1}_i$ is not nonnegative.

For the further study, it is convenient to introduce the operators
$$\mathcal{J}_{i,\mathbf{n}}(X,Y)=\frac{\mu_i g_i n_i}
{\big<\mathbf{n},\mathbf{g}\big>}X-\lambda_iY,
$$
where $X$ and $Y$ are real numbers. Then, \eqref{eq-c2} can be
rewritten in the form
\begin{equation}\label{eq-c4}
\sum_{i=1}^I\left[\mathcal{J}_{i,\mathbf{n}+\mathbf{1}_i}(P_{\mathbf{n}+\mathbf{1}_i},
P_{\mathbf{n}})-
\mathcal{J}_{i,\mathbf{n}}(P_{\mathbf{n}},P_{\mathbf{n}-\mathbf{1}_i})\right]=0.
\end{equation}

\subsection{Proof of Theorem \ref{thm-3}} 
We first obtain properties that the function
$F(\mathbf{n},\mathbf{g})$ must satisfy to be a solution of
\eqref{eq-d1}.  Substituting \eqref{eq-d1} into \eqref{eq-c2} and
canceling the (non-zero) factor $\prod_{l=1}^I\rho_l^{n_l}$ yields
\begin{align}\label{eq-d3}
&\sum_{i=1}^I\frac{\mu_ig_in_i}{\big<\mathbf{n},\mathbf{g}\big>}F(\mathbf{n},\mathbf{g})
-
\sum_{i=1}^I\frac{\mu_ig_i(n_i+1)}{\big<\mathbf{n}+\mathbf{1}_i,\mathbf{g}\big>}
F(\mathbf{n}+\mathbf{1}_i,\mathbf{g})\rho_i
\nonumber\\
&=\sum_{i=1}^I\lambda_i{F}(\mathbf{n}-\mathbf{1}_i,\mathbf{g})\frac{1}{\rho_i}
-\sum_{i=1}^I\lambda_iF(\mathbf{n},\mathbf{g}),
\end{align}
where ${F}(\mathbf{n}-\mathbf{1}_i,\mathbf{g})=0$ in the case where
the vector $\mathbf{n}-\mathbf{1}_i$ is \textit{not} nonnegative.
Equation \eqref{eq-d3} is equivalent to
\begin{align*}\label{eq-d4*}
&\sum_{i=1}^I\frac{\mu_ig_in_i}{\big<\mathbf{n},\mathbf{g}\big>}F(\mathbf{n},\mathbf{g})
-
\sum_{i=1}^I\lambda_i\frac{g_i(n_i+1)}{\big<\mathbf{n}+\mathbf{1}_i,\mathbf{g}\big>}
F(\mathbf{n}+\mathbf{1}_i,\mathbf{g})
\\
&=\sum_{i=1}^I\mu_i{F}(\mathbf{n}-\mathbf{1}_i,\mathbf{g})
-\sum_{i=1}^I\lambda_iF(\mathbf{n},\mathbf{g}). \nonumber
\end{align*}
Since this must hold for all values of $\lambda_i$ and all values of
$\mu_i$, we can equate the coefficients of $\lambda_i$ and $\mu_i$
to obtain that $F(\mathbf{n},\mathbf{g})$ must, for all $i$ and all
nonnegative vectors $\mathbf{n}$, satisfy the recurrence relation
\begin{equation}\label{eq-d10}
  F(\mathbf{n}+\mathbf{1}_i,\mathbf{g})
= \frac{\big<\mathbf{n}+\mathbf{1}_i,\mathbf{g}\big>}{(n_i+1)g_i}
  F(\mathbf{n},\mathbf{g}),
\end{equation}
where $F(\mathbf{0},\mathbf{g})$ is the initial positive value for
the recurrence relation of \eqref{eq-d10}.

We prove now that representation \eqref{eq-d1} is correct if and
only if the components of the vector $\mathbf{g}$ all are equal. Set
$F(\mathbf{0},\mathbf{g})=C$, where $C$ is a positive constant
depending on the vector $\mathbf{g}$, and assume, to obtain a
contradiction, that there exist $i$ and $l$ such that $g_i\neq g_l$.
By \eqref{eq-d10}, for $F(\mathbf{n},\mathbf{g})$ to satisfy
\eqref{eq-d1}, we require $
F(\mathbf{1}_i,\mathbf{g})=F(\mathbf{0},\mathbf{g})\frac{g_i}{g_i}=C,
$ and $
F(\mathbf{1}_i+\mathbf{1}_l,\mathbf{g})=F(\mathbf{1}_i,\mathbf{g})
\frac{g_i+g_l}{g_l}=C\frac{g_i+g_l}{g_l}. $ On the other hand, by
the similar way we obtain
$F(\mathbf{1}_i+\mathbf{1}_l,\mathbf{g})=C\frac{g_i+g_l}{g_i}$ if we
first find $F(\mathbf{1}_l,\mathbf{g})=C$, and then
$F(\mathbf{1}_i+\mathbf{1}_l,\mathbf{g})=F(\mathbf{1}_l,\mathbf{g})
\frac{g_i+g_l}{g_i}=C\frac{g_i+g_l}{g_i}$. This is only correct when
$g_i=g_l$ and, hence, it contradicts to the assumption that $g_i\neq
g_l$. Hence, the function $F(\mathbf{0},\mathbf{g})$ is not uniquely
defined if the vector $\mathbf{g}$ has distinct components.

So, we arrived at the contradiction, which proves that the only
equality $g_i=g_l$ must be valid. Hence, the function
$F(\mathbf{n},\mathbf{g})$ is well-defined if and only if the vector
$\mathbf{g}$ has identical components. This finishes the proof.

\begin{rem} It is readily seen from \eqref{eq-c2} that
the case $g_1=g_2=\ldots=g_I\equiv g$ reduces to the case
$g_1=g_2=\ldots=g_I\equiv 1$, which corresponds to Egalitarian PS
systems. Here we have
$$
F(\mathbf{n},\mathbf{1})=\frac{|\mathbf{n}|!}{n_1!n_2!\cdot\ldots\cdot
n_I!},
$$
and for the stationary distributions we have:
\begin{equation}\label{eq-d4}
P_{\mathbf{n}}=(1-\rho)\frac{|\mathbf{n}|!}{n_1!n_2!\cdot\ldots\cdot
n_I!} \prod_{i=1}^I\rho_i^{n_i}, \quad \mathbf{n}>\mathbf{0}.
\end{equation}
\end{rem}

\begin{rem} It follows from Theorem \ref{thm-3} that
in the only case of Egalitarian PS systems the equality
$\mathcal{J}_{i,\mathbf{n}}(P_{\mathbf{n}},P_{\mathbf{n}-\mathbf{1}_i})=0$
holds for all $i=1,2,\ldots, I$ and all $\mathbf{n}\in\mathcal{N}$.
\end{rem}

\subsection{Proof of Theorem \ref{thm-4} and Corollary \ref{cor-1}}

The proof of Theorem \ref{thm-4} is divided into auxiliary lemmas.
First, we introduce the concepts and notation that are used to prove
Theorem \ref{thm-4}.

\subsubsection{Concepts and notation.}\label{CaN}

For any vector $\mathbf{n}>\mathbf{0}$, let us present the elements
of the vector $\mathbf{g}$ in the following two orders
\begin{equation}\label{eq-g1}
\underbrace{g_1,~g_1,~\ldots,~g_1;}_{n_1 \ \text{times}}
\underbrace{g_2,~g_2,~\ldots,~g_2;}_{n_2 \ \text{times}}\ldots;
\underbrace{g_I,~g_I,~\ldots,~g_I}_{n_I \ \text{times}},
\end{equation}
\begin{equation}\label{eq-g1.1}
\underbrace{g_I,~g_I,~\ldots,~g_I;}_{n_I \ \text{times}}
\underbrace{g_{I-1},~g_{I-1},~\ldots,~g_{I-1};}_{n_{I-1} \
\text{times}}\ldots; \underbrace{g_1,~g_1,~\ldots,~g_1}_{n_1 \
\text{times}}.
\end{equation}
The order in \eqref{eq-g1} is called \textit{forward}, and the order
in \eqref{eq-g1.1} is called \textit{backward}.

For the forward order, denote the sequence of the partial sums by
\begin{equation}\label{eq-d6}
S_{\mathbf{n},1}^{(1)}=g_1, S_{\mathbf{n},2}^{(1)}=2g_1,\ldots,
S_{\mathbf{n},n_1}^{(1)}=n_1g_1,
S_{\mathbf{n},n_1+1}^{(1)}=n_1g_1+g_2, \ldots,
S_{\mathbf{n},|\mathbf{n}|}^{(1)}=\big<\mathbf{n},\mathbf{g}\big>,
\end{equation}
and for the backward order, the sequence of partial sums is denoted
by
\begin{equation}\label{eq-d7}
\begin{aligned}
&S_{\mathbf{i},1}^{(2)}=g_K, S_{\mathbf{i},2}^{(2)}=2g_K,\ldots,
S_{\mathbf{n},n_I}^{(2)}=n_Ig_I,
S_{\mathbf{n},n_I+1}^{(2)}=n_Ig_I+g_{I-1}, \ldots,
S_{\mathbf{n},|\mathbf{n}|}^{(2)}=\big<\mathbf{n},\mathbf{g}\big>.
\end{aligned}
\end{equation}

Introduce the probability mass functions $P_{\mathbf{n}}^{(1)}$ and
$P_{\mathbf{n}}^{(2)}$ as follows:
\begin{equation}\label{eq-d8}
P^{(1)}_{\mathbf{n}}=C^{(1)}\frac{\prod_{l=1}^{|\mathbf{n}|}
S_{\mathbf{n},l}^{(1)}}{n_1!n_2!\cdot\ldots\cdot
n_I!}\prod_{i=1}^I\left(\frac{\rho_i}{g_i}\right)^{n_i},
\end{equation}
and
\begin{equation}\label{eq-d9}
P^{(2)}_{\mathbf{n}}=C^{(2)}
\frac{\prod_{l=1}^{|\mathbf{n}|}S_{\mathbf{n},l}^{(2)}}{n_1!n_2!\cdot\ldots
\cdot n_I!}\prod_{i=1}^I\left(\frac{\rho_i}{g_i}\right)^{n_i},
\end{equation}
where the normalization constants $C^{(1)}$ and $C^{(2)}$ are
$$
C^{(1)}=
\left[\sum_{\mathbf{n}\geq\mathbf{0}}\frac{\prod_{l=1}^{|\mathbf{n}|}
S_{\mathbf{n},l}^{(1)}}{n_1!n_2!\cdot\ldots\cdot n_I!
}\prod_{i=1}^I\left(\frac{\rho_i}{g_i}\right)^{n_i} \right]^{-1},
$$
and
$$
C^{(2)}=
\left[\sum_{\mathbf{n}\geq\mathbf{0}}\frac{\prod_{l=1}^{|\mathbf{n}|}
S_{\mathbf{n},l}^{(2)}}{n_1!n_2!\cdot\ldots\cdot n_I!
}\prod_{i=1}^I\left(\frac{\rho_i}{g_i}\right)^{n_i} \right]^{-1}.
$$

Apparently,
$$
\frac{\prod_{l=1}^{|\mathbf{n}|}
S_{\mathbf{n},l}^{(1)}}{n_1!n_2!\cdot\ldots\cdot
n_I!}\prod_{i=1}^I\left(\frac{\rho_i}{g_i}\right)^{n_i}\leq
\frac{|\mathbf{n}|!}{n_1!n_2!\cdot\ldots\cdot
n_I!}\prod_{i=1}^I{\rho_i}^{n_i},
$$
for all $\mathbf{n}\geq\mathbf{0}$. So, if $\rho<1$, then
$P_{\mathbf{n}}^{(1)}$ is a proper probability mass function with
the normalization constant $C^{(1)}$ satisfying the inequality
$C^{(1)}\geq1-\rho.$ However, $P_{\mathbf{n}}^{(2)}$  is a proper
probability mass function only in the case when the series
\begin{equation}\label{eq-d13}
\sum_{\mathbf{n}\geq\mathbf{0}}\frac{\prod_{l=1}^{|\mathbf{n}|}
S_{\mathbf{n},l}^{(2)}}{n_1!n_2!\cdot\ldots\cdot n_I!
}\prod_{i=1}^I\left(\frac{\rho_i}{g_i}\right)^{n_i}
\end{equation}
converges. We cannot claim the convergence of \eqref{eq-d13} in
general.

\subsubsection{Auxiliary lemmas.}\label{Auxiliary} In Lemmas \ref{lem-2}
and \ref{lem-5} given below, it is assumed that
$P_{\mathbf{n}}^{(2)}$ is a proper probability mass function.

\begin{lem}\label{lem-2} For all $\mathbf{n}\in\mathcal{N}$
we have the following relations:
\begin{eqnarray}
\mathcal{J}_{I,\mathbf{n}}(P_{\mathbf{n}}^{(1)},P_{\mathbf{n}-\mathbf{1}_I}^{(1)})&=&0,
\label{eq-l2}\\
\mathcal{J}_{i,\mathbf{n}}(P_{\mathbf{n}}^{(1)},P_{\mathbf{n}-\mathbf{1}_i}^{(1)})&<&0,
\ i=1,2\ldots,I-1,\label{eq-l3}\\
\mathcal{J}_{1,\mathbf{i}}(P_{\mathbf{n}}^{(2)},P_{\mathbf{n}-\mathbf{1}_1}^{(2)})&=&0,\label{eq-l4}\\
\mathcal{J}_{i,\mathbf{n}}(P_{\mathbf{n}}^{(2)},P_{\mathbf{n}-\mathbf{1}_i}^{(2)})&>&0,
\ i=2,3\ldots,I,\label{eq-l5}
\end{eqnarray}
\end{lem}

\begin{proof} For better readability, we provide the proof of this lemma
for shifted indices by replacing
$\mathcal{J}_{i,\mathbf{n}}(P_{\mathbf{n}}^{(1)},P_{\mathbf{n}-\mathbf{1}_i}^{(1)})$
with
$\mathcal{J}_{i,\mathbf{n}+\mathbf{1}_i}(P_{\mathbf{n}+\mathbf{1}_i}^{(1)},
P_{\mathbf{n}}^{(1)})$ ($i=1,2,\ldots,I$). For instance, instead of
\eqref{eq-l2} we prove
$\mathcal{J}_{I,\mathbf{n}+\mathbf{1}_I}(P_{\mathbf{n}+\mathbf{1}_I}^{(1)},
P_{\mathbf{n}}^{(1)})=0.$

Relations \eqref{eq-l2} and \eqref{eq-l4} follow by the direct
substitution, since for the first $|\mathbf{n}|$ partial sums we
have $S^{(1)}_{\mathbf{n}+\mathbf{1}_I,l}=S^{(1)}_{\mathbf{n},l}$
and, respectively, $S^{(2)}_{\mathbf{n}+\mathbf{1}_1,l}
=S_{\mathbf{n},l}^{(2)}$ ($l=1,2,\ldots,|\mathbf{n}|$), and hence,
$$
\frac{1}{\big<\mathbf{n}+\mathbf{1}_I,\mathbf{g}\big>}\prod_{l=1}^{|\mathbf{n}|+1}
S^{(1)}_{\mathbf{n}+\mathbf{1}_I,l}=\prod_{l=1}^{|\mathbf{n}|}S_{\mathbf{n},l}^{(1)},
$$
and
$$
\frac{1}{\big<\mathbf{n}+\mathbf{1}_1,\mathbf{g}\big>}\prod_{l=1}^{|\mathbf{n}|+1}
S^{(2)}_{\mathbf{n}+\mathbf{1}_1,l}=\prod_{l=1}^{|\mathbf{n}|}S_{\mathbf{n},l}^{(2)}.
$$

To prove the strong inequality of \eqref{eq-l3} note, that in the
relation
\begin{equation*}\label{eq-e4}
P^{(1)}_{\mathbf{n}+\mathbf{1}_i}=C^{(1)}\frac{\prod_{l=1}^{|\mathbf{n}|+1}
S_{\mathbf{n}+\mathbf{1}_i,l}^{(1)}}{n_1!n_2!\cdot\ldots\cdot
n_{i-1}!(n_i+1)!n_{i+1}! \cdot\ldots\cdot
n_I!}\left(\frac{\rho_i}{g_i}\right)\prod_{l=1}^I\left(\frac{\rho_l}{g_l}\right)^{n_l},
\end{equation*}
the product term
$\prod_{l=1}^{|\mathbf{n}|+1}S_{\mathbf{n}+\mathbf{1}_i,l}^{(1)}$
contains the following $|\mathbf{n}|+1$ terms:
\begin{equation}\label{eq-e5}
\begin{aligned}
&S_{\mathbf{n}+\mathbf{1}_i,1}^{(1)}=g_1,S_{\mathbf{n}+\mathbf{1}_i,2}^{(1)}=2g_1,
\ldots, S_{\mathbf{n}+\mathbf{1}_i,n_1}^{(1)}=n_1g_1,
S_{\mathbf{n}+\mathbf{1}_i,n_1+1}=
n_1g_1+g_2,\ldots,\\
&S_{\mathbf{n}+\mathbf{1}_i,|\mathbf{n}|_i}^{(1)}=n_1g_1+\ldots+n_ig_i,\\
&S_{\mathbf{n}+\mathbf{1}_i,|\mathbf{n}|_i+1}^{(1)}=n_1g_1+\ldots+(n_i+1)g_i,\\
&S_{\mathbf{n}+\mathbf{1}_i,|\mathbf{n}|_i+2}^{(1)}=n_1g_1+\ldots+(n_i+1)g_i+g_{i+1},\\
&\ldots,\\
&S_{\mathbf{n}+\mathbf{1}_i,|\mathbf{n}|}^{(1)}
=\big<\mathbf{n}+\mathbf{1}_i,\mathbf{g}\big> \ \text{minus the last
element in sequence}
 \ \eqref{eq-g1},\\
&S_{\mathbf{n}+\mathbf{1}_i,|\mathbf{n}|+1}^{(1)}=
\big<\mathbf{n}+\mathbf{1}_i,\mathbf{g}\big>,
\end{aligned}
\end{equation}
and after dividing the term
$\prod_{l=1}^{|\mathbf{n}|+1}S_{\mathbf{n}+\mathbf{1}_i,l}^{(1)}$ by
$\big<\mathbf{n}+\mathbf{1}_i,\mathbf{g}\big>$, the last term in
\eqref{eq-e5} disappears.

Let us compare the product terms in \eqref{eq-d6} and the first
$|\mathbf{n}|$ terms in \eqref{eq-e5}. The first
$n_1+n_2+\ldots+n_i$ product terms in \eqref{eq-d6} and
\eqref{eq-e5} coincide. However, for all of the following terms we
have $S_{\mathbf{n},l}^{(1)}>S_{\mathbf{n}+\mathbf{1}_i,l}^{(1)}$,
$l=n_1+n_2+\ldots+n_i+1,\ldots,|\mathbf{n}|$. For example,
$$
\begin{aligned}
S_{\mathbf{n},n_1+n_2+\ldots+n_i+1}^{(1)}&=n_1g_1+\ldots+n_ig_i+g_{i+1}\\
&>n_1g_1+\ldots+n_ig_i+g_i\\
&=S_{\mathbf{n}+\mathbf{1}_i,n_1+n_2+\ldots+n_i+1}^{(1)},
\end{aligned}
$$
since by the assumption of the theorem $g_{i+1}>g_i$. Henceforth,
after algebraic reductions we obtain \eqref{eq-l3}. The proof of the
strong inequality of \eqref{eq-l5} is similar. Lemma \ref{lem-2} is
proved.
\end{proof}

\begin{lem}\label{lem-3} Let $\gamma_1$, $\gamma_2$, \ldots,
$\gamma_I$ be positive numbers ($\sum_{i=1}^I \gamma_i=1$), let
$\theta_i^{(1)}$ and $\theta_i^{(2)}$ be the values that are defined
by \eqref{eq-d11} and \eqref{eq-d12}, let $\Delta_i^{(1)}$ and
$\Delta_i^{(2)}$ be the values that are defined by \eqref{eq-d14}
and \eqref{eq-d15}, and let Condition \eqref{eq-l7} be satisfied.

Then the limiting, as $N\to\infty$, stationary probabilities
$P_{\lfloor N{\bf\Gamma}\rfloor}^{(2)}$ are well-defined, and
\begin{eqnarray}
\lim_{N\to\infty}\frac{P_{\lfloor
N{\bf\Gamma}\rfloor+\mathbf{1}_i}^{(1)}} {P_{\lfloor
N{\bf\Gamma}\rfloor}^{(1)}}&=&\Delta_i^{(1)},\label{eq-e20}\\
\lim_{N\to\infty}\frac{P_{\lfloor
N{\bf\Gamma}\rfloor+\mathbf{1}_i}^{(2)}} {P_{\lfloor
N{\bf\Gamma}\rfloor}^{(2)}}&=&\Delta_i^{(2)}.\label{eq-e21}
\end{eqnarray}
\end{lem}

\begin{proof}
Indeed,  for
$\frac{P_{\mathbf{n}+\mathbf{1}_i}^{(1)}}{P_{\mathbf{n}}^{(1)}}$ we
have as follows:
\begin{equation}\label{eq-h1}
\frac{P_{\mathbf{n}+\mathbf{1}_i}^{(1)}}{P_{\mathbf{n}}^{(1)}}=
\frac{\rho_i}{(n_i+1)g_i}\cdot
\frac{\prod_{l=1}^{|\mathbf{n}|+1}S_{\mathbf{n}+\mathbf{1}_i,l}^{(1)}}
{\prod_{l=1}^{|\mathbf{n}|}S_{\mathbf{n},l}^{(1)}}=
\frac{\rho_i\big<\mathbf{n}+\mathbf{1}_i,\mathbf{g}\big>}{(n_i+1)g_i}\cdot
\prod_{l=1}^{|\mathbf{n}|}\frac{S_{\mathbf{n}+\mathbf{1}_i,l}^{(1)}}
{S_{\mathbf{n},l}^{(1)}}.
\end{equation}
Assuming that $N$ tends to infinity in \eqref{eq-h1}, then for
$i=1,2,\ldots, I$ we have the expansion
\begin{equation}\label{eq-h2}
\begin{aligned}
\frac{\rho_i\big<\lfloor
N{\bf\Gamma}\rfloor+\mathbf{1}_i,\mathbf{g}\big>}{(\lfloor
N\gamma_i\rfloor+1)g_i}&= \frac{\rho_i(\lfloor N\gamma_1\rfloor
g_1+\lfloor N\gamma_2\rfloor g_2+\ldots+\lfloor N\gamma_I\rfloor
g_I)+\rho_ig_i}{\lfloor
N\gamma_i\rfloor g_i+g_i}\\
&=\frac{\rho_i(\gamma_1g_1+\gamma_2g_2+\ldots+\gamma_Ig_I)}{\gamma_ig_i}
\left[1+O\left(\frac{1}{N}\right)\right].
\end{aligned}
\end{equation}
Next, with the aid of \eqref{eq-d6} and \eqref{eq-e5} we prove
\begin{equation}\label{eq-h3}
\begin{aligned}
&\lim_{N\to\infty}\prod_{l=1}^{|\lfloor
N{\bf\Gamma}\rfloor|}\frac{S_{\lfloor
N{\bf\Gamma}\rfloor+\mathbf{1}_i,l}^{(1)}} {S_{\lfloor
N{\bf\Gamma}\rfloor,l}^{(1)}}=\theta_i^{(1)}.
\end{aligned}
\end{equation}

Indeed, for $l=1,2,\ldots,\lfloor N\gamma_1\rfloor+\lfloor
N\gamma_2\rfloor+\ldots+\lfloor N\gamma_i\rfloor$, we have
$S_{\lfloor N{\bf\Gamma}\rfloor+\mathbf{1}_i,l}^{(1)}=S_{\lfloor
N{\bf\Gamma}\rfloor,l}^{(1)}$, and hence
\begin{equation*}
\lim_{N\to\infty}\prod_{l=1}^{\lfloor
N\gamma_{1}\rfloor+\ldots+\lfloor
N\gamma_{i}\rfloor}\frac{S_{\lfloor
N{\bf\Gamma}\rfloor+\mathbf{1}_i,l}^{(1)}}{S_{\lfloor
N{\bf\Gamma}\rfloor,l}^{(1)}}=1.
\end{equation*}
For further simplifications, we use the conventional notation
$$
\big|\lfloor N{\bf\Gamma}\rfloor\big|_i=\lfloor
N\gamma_1\rfloor+\lfloor N\gamma_2\rfloor+\ldots+\lfloor
N\gamma_i\rfloor.
$$
Let us first find $\lim_{N\to\infty}\prod_{l=|\lfloor
N{\bf\Gamma}\rfloor|_i+1}^{|\lfloor N{\bf\Gamma}\rfloor|_{i+1}}
\frac{S_{\lfloor
N{\bf\Gamma}\rfloor+\mathbf{1}_i,l}^{(1)}}{S_{\lfloor
N{\bf\Gamma}\rfloor,l}^{(1)}}$.

 Notice, that
for any $1\leq m\leq\lfloor N\gamma_{k+1}\rfloor$, we have
\begin{equation}\label{eq-h3.1}
\begin{aligned}
&\frac{S_{\lfloor N{\bf\Gamma}\rfloor+\mathbf{1}_i,|\lfloor
N{\bf\Gamma}\rfloor|_i+m}^{(1)}} {S_{\lfloor
N{\bf\Gamma}\rfloor,|\lfloor
N{\bf\Gamma}\rfloor|_i+m}^{(1)}}\\
&=\frac{g_1\lfloor N\gamma_1\rfloor+g_2\lfloor
N\gamma_2\rfloor+\ldots+g_i\lfloor
N\gamma_i\rfloor+g_i+(m-1)g_{i+1}}{g_1\lfloor
N\gamma_1\rfloor+g_2\lfloor N\gamma_2\rfloor+\ldots+g_i\lfloor
N\gamma_i\rfloor+g_i+(m-1)g_{i+1}+(g_{i+1}-g_i)}\\
&=1-\frac{g_{i+1}-g_i}{g_1\lfloor N\gamma_1\rfloor+g_2\lfloor
N\gamma_2\rfloor+\ldots+g_i\lfloor
N\gamma_i\rfloor+(m-1)g_{i+1}+g_i}.
\end{aligned}
\end{equation}
Hence,
\begin{equation}\label{eq-h5}
\begin{aligned}
\lim_{N\to\infty}\prod_{l=|\lfloor
N{\bf\Gamma}\rfloor|_i+1}^{|\lfloor N{\bf\Gamma}\rfloor|_{i+1}}
\frac{S_{\lfloor
N{\bf\Gamma}\rfloor+\mathbf{1}_i,l}^{(1)}}{S_{\lfloor
N{\bf\Gamma}\rfloor,l}^{(1)}}
&=\exp\left(-\int_0^1\frac{(g_{i+1}-g_i)\gamma_{i+1}}
{\gamma_1g_1+\gamma_2g_2+\ldots+\gamma_{i}g_i+\gamma_{i+1}g_{i+1}x}\mathrm{d}x\right)\\
&=\exp\left[\left(\frac{g_i}{g_{i+1}}-1\right)
\ln\left(1+\frac{\gamma_{i+1}g_{i+1}}{\sum_{k=1}^i\gamma_kg_k}\right)\right].
\end{aligned}
\end{equation}

Similarly, for any $j=0,1,\ldots, I-i$ we have
\begin{equation}\label{eq-h6}
\begin{aligned}
&\lim_{N\to\infty}\prod_{l=|\lfloor
N{\bf\Gamma}\rfloor|_{i+j}+1}^{|\lfloor
N{\bf\Gamma}\rfloor|_{i+j+1}} \frac{S_{\lfloor
N{\bf\Gamma}\rfloor+\mathbf{1}_i,l}^{(1)}} {S_{\lfloor
N{\bf\Gamma}\rfloor,l}^{(1)}}\\
&=\exp\left(-\int_0^1\frac{(g_{i+1+j}-g_i)\gamma_{i+1+j}}
{\gamma_1g_1+\gamma_2g_2+\ldots+\gamma_{i+j}g_{i+j}+\gamma_{i+j+1}g_{i+j+1}x}
\mathrm{d}x\right)\\
&=\exp\left[\left(\frac{g_i}{g_{i+1+j}}-1\right)
\ln\left(1+\frac{\gamma_{i+1+j}g_{i+1+j}}{\sum_{k=1}^{i+j}\gamma_kg_k}\right)\right].
\end{aligned}
\end{equation}
So, \eqref{eq-h3} follows.

Again, we have \eqref{eq-h2}, and similarly to \eqref{eq-h3} we
obtain
\begin{equation}\label{eq-h33}
\lim_{N\to\infty}\prod_{l=1}^{|\lfloor
N{\bf\Gamma}\rfloor|}\frac{S_{\lfloor
N{\bf\Gamma}\rfloor+\mathbf{1}_i,l}^{(2)}} {S_{\lfloor
N{\bf\Gamma}\rfloor,l}^{(2)}} =\theta_i^{(2)}.
\end{equation}
Then Relations \eqref{eq-h2}, \eqref{eq-h33} and Condition
\eqref{eq-l7} make the stationary probabilities
$P_{\mathbf{n}}^{(2)}$ well-defined, since according to these
relations, $\lim_{N\to\infty}\frac{P_{\lfloor
N{\bf\Gamma}\rfloor+\mathbf{1}_i}^{(2)}}{P_{\lfloor
N{\bf\Gamma}\rfloor}^{(2)}}<1$, $i=1,2,\ldots,I$. The last also
means that the series in \eqref{eq-d13} converges, if the infinite
sum is taken on the set of indices specified by the vectors of
$\mathcal{N}_{\bf\Gamma}$. Hence, relations \eqref{eq-e20} and
\eqref{eq-e21} follow. The lemma is proved.
\end{proof}

For the purpose of this paper, we need in stronger results than
those are given by Lemma \ref{lem-3}. The following lemma is an
extension of Lemma \ref{lem-3}.

\begin{lem}\label{lem-3'} Under the assumptions of Lemma
\ref{lem-3}, as $N\to\infty$, for $i=1,2,\ldots, I$ we have:
\begin{equation}\label{eq-h27}
\begin{aligned}
P_{\lfloor N{\bf\Gamma}\rfloor+\mathbf{1}_i}^{(1)}
=&\Delta_i^{(1)}P_{\lfloor
N{\bf\Gamma}\rfloor}^{(1)}\left[1-\frac{\alpha_i^{(1)}}{N}+
o\left(\frac{1}{N}\right)\right],
\end{aligned}
\end{equation}
and
\begin{equation}\label{eq-h30}
\begin{aligned}
P_{\lfloor N{\bf\Gamma}\rfloor+\mathbf{1}_i}^{(2)}
=&\Delta_i^{(2)}P_{\lfloor
N{\bf\Gamma}\rfloor}^{(2)}\left[1-\frac{\alpha_i^{(2)}}{N}+
o\left(\frac{1}{N}\right)\right],
\end{aligned}
\end{equation}
where
\begin{equation}\label{eq-h17}
\begin{aligned}
\alpha_i^{(1)}=&\left(\frac{\sum_{j=1}^I\gamma_jg_j-\gamma_ig_i}
{\gamma_i\sum_{j=1}^I\gamma_jg_j}\right)+\sum_{j=1}^{I-i}\frac{1}{2\gamma_{i+j}}
\left(\frac{g_i}{g_{i+j}}-1\right)^2
\left[\ln\left(1+\frac{\gamma_{i+j}g_{i+j}}{\sum_{k=1}^{i+j}\gamma_kg_k}\right)
\right]^2\\
&\times\exp\left[\left(\frac{g_i}{g_{i+j}}-1\right)
\ln\left(1+\frac{\gamma_{i+1}g_{i+1}}{\sum_{k=1}^{i+j}\gamma_kg_k}\right)\right],
\quad i=1,2,\ldots,I-1,
\end{aligned}
\end{equation}
\begin{equation}\label{eq-h18}
\alpha_I^{(1)}=\left(\frac{\sum_{j=1}^I\gamma_jg_j-\gamma_Ig_I}
{\gamma_I\sum_{j=1}^I\gamma_jg_j}\right),
\end{equation}
\begin{equation}\label{eq-h19}
\alpha_1^{(2)}=\left(\frac{\sum_{j=1}^I\gamma_jg_j-\gamma_1g_1}
{\gamma_1\sum_{j=1}^I\gamma_jg_j}\right),
\end{equation}
and
\begin{equation}\label{eq-h20}
\begin{aligned}
\alpha_i^{(2)}=&\left(\frac{\sum_{j=1}^I\gamma_jg_j-\gamma_ig_i}
{\gamma_i\sum_{j=1}^I\gamma_jg_j}\right)+\sum_{j=1}^{i-1}\frac{1}{2\gamma_j}
\left(\frac{g_i}{g_j}-1\right)^2\left[\ln
\left(1+\frac{\gamma_jg_j}{\sum_{k=j+1}^I\gamma_kg_k}\right)\right]^2\\
&\times\exp\left[\left(\frac{g_i}{g_j}-1\right)\ln\left(1+
\frac{\gamma_jg_j}{\sum_{k=j+1}^I\gamma_kg_k}\right)\right], \quad
i=2,3,\ldots,I.
\end{aligned}
\end{equation}
\end{lem}

\begin{proof} We are to establish the exact values of constants
$c_i$, $c_{\theta_i}^{(1)}$, and $c_{\theta_i}^{(2)}$
($i=1,2,\ldots,I$) in the expansions
$$
\frac{\rho_i\big<\lfloor
N{\bf\Gamma}\rfloor+\mathbf{1}_i,\mathbf{g}\big>}{(\lfloor
N\gamma_i\rfloor+1)g_i}=
\frac{\rho_i\sum_{j=1}^I\gamma_jg_j}{\gamma_ig_i}\left[1+\frac{c_i}{N}+
O\left(\frac{1}{N^2}\right)\right],
$$
$$
\prod_{l=1}^{N}\frac{S^{(1)}_{\lfloor
N{\bf\Gamma}\rfloor+\mathbf{1}_i,l}} {S^{(1)}_{\lfloor
N{\bf\Gamma}\rfloor,l}}=\theta_i^{(1)}\left[1+\frac{c_{\theta_i}^{(1)}}{N}
+o\left(\frac{1}{N}\right)\right]
$$
and
$$
\prod_{l=1}^{N}\frac{S^{(2)}_{\lfloor
N{\bf\Gamma}\rfloor+\mathbf{1}_i,l}} {S^{(2)}_{\lfloor
N{\bf\Gamma}\rfloor,l}}=\theta_i^{(2)}\left[1+\frac{c_{\theta_i}^{(2)}}{N}
+o\left(\frac{1}{N}\right)\right]
$$
for large $N$. Then, we will arrive at necessary expansions
$$
P_{\lfloor
N{\bf\Gamma}\rfloor+\mathbf{1}_i}^{(1)}=\Delta_{i}^{(1)}P_{\lfloor
N{\bf\Gamma}\rfloor}^{(1)} \left[1+\frac{c_i}{N}
+\frac{c_{\theta_i}^{(1)}}{N}+o\left(\frac{1}{N}\right)\right]
$$
and
$$
P_{\lfloor
N{\bf\Gamma}\rfloor+\mathbf{1}_i}^{(2)}=\Delta_{i}^{(2)}P_{\lfloor
N{\bf\Gamma}\rfloor}^{(2)} \left[1+\frac{c_i}{N}
+\frac{c_{\theta_i}^{(2)}}{N}+o\left(\frac{1}{N}\right)\right].
$$
From \eqref{eq-h2} one can obtain a more precise expansion than that
is given by the right-hand side of \eqref{eq-h2}. Namely, after some
algebra,
\begin{equation}\label{eq-h34}
\begin{aligned}
\frac{\rho_i\big<\lfloor
N{\bf\Gamma}\rfloor+\mathbf{1}_i,\mathbf{g}\big>}{(\lfloor
N\gamma_i\rfloor+1)g_i}&=
\frac{\rho_i\sum_{j=1}^I\gamma_jg_j}{\gamma_ig_i}
\left[1+\frac{1}{N}\left(\frac{\gamma_ig_i-\sum_{j=1}^I\gamma_jg_j}
{\gamma_i\sum_{j=1}^I\gamma_jg_j}\right)+o\left(\frac{1}{N}\right)\right].
\end{aligned}
\end{equation}
So, the constant $c_i$ is found, and from this estimate we
immediately arrive at the estimates \eqref{eq-h27} and
\eqref{eq-h30} for $i=I$ (containing the constant $\alpha_I^{(1)}$)
and $i=1$ (containing the constant $\alpha_1^{(2)}$), respectively.

Find now the constants $c_{\theta_i}^{(1)}$, $i=1,2,\ldots,I-1$, and
thus prove the estimate \eqref{eq-h27} for $i=1,2,\ldots,I-1$. From
\eqref{eq-h3.1}, for large $N$ using the mean value theorem, for
some value $\eta_N\in(0,1)$ we obtain
\begin{equation}\label{eq-h35}
\begin{aligned}
\prod_{l=|\lfloor N{\bf\Gamma}\rfloor|_i+1}^{|\lfloor
N{\bf\Gamma}\rfloor|_{i+1}}\frac{S_{\lfloor
N{\bf\Gamma}\rfloor+\mathbf{1}_i,l}^{(1)}} {S_{\lfloor
N{\bf\Gamma}\rfloor,l}^{(1)}}&=\left(1-\frac{g_{i+1}-g_i}
{\sum_{j=1}^ig_j\lfloor N\gamma_j\rfloor+\eta_N g_{i+1}\lfloor
N\gamma_{i+1}\rfloor}\right)^{\lfloor N\gamma_{i+1}\rfloor}.
\end{aligned}
\end{equation}
Then, the integral given by \eqref{eq-h6} can be written in the form
\begin{equation}\label{eq-h36}
\lim_{N\to\infty}\prod_{l=|\lfloor
N{\bf\Gamma}\rfloor|_i+1}^{|\lfloor N{\bf\Gamma}\rfloor|_{i+1}}
\frac{S_{\lfloor N{\bf\Gamma}\rfloor+\mathbf{1}_i,l}^{(1)}}
{S_{\lfloor
N{\bf\Gamma}\rfloor,l}^{(1)}}=\exp\left(-\frac{(g_{i+1}-g_i)\gamma_{i+1}}
{\sum_{j=1}^ig_j\gamma_j+\eta g_{i+1}\gamma_{j+1}}\right),
\end{equation}
where $\eta=\lim_{N\to\infty}\eta_N$. Let us find the limit
\begin{equation}\label{eq-h37}
\lim_{N\to\infty}N\left[\prod_{l=|\lfloor
N{\bf\Gamma}\rfloor|_i+1}^{|\lfloor
N{\bf\Gamma}\rfloor|_{i+1}}\frac{S_{\lfloor
N{\bf\Gamma}\rfloor+\mathbf{1}_i,l}^{(1)}} {S_{\lfloor
N{\bf\Gamma}\rfloor,l}^{(1)}}-\exp\left(-\frac{(g_{i+1}-g_i)\gamma_{i+1}}
{\sum_{j=1}^ig_j\gamma_j+\eta g_{i+1}\gamma_{j+1}}\right)\right].
\end{equation}
Expanding the right-hand side of \eqref{eq-h35} we obtain
\begin{equation}
\begin{aligned}
&\left(1-\frac{g_{i+1}-g_i} {\sum_{j=1}^ig_j\lfloor
N\gamma_j\rfloor+\eta_N g_{i+1}\lfloor
N\gamma_{i+1}\rfloor}\right)^{\lfloor N\gamma_{i+1}\rfloor}
=\exp\left(-\frac{(g_{i+1}-g_i)\gamma_{i+1}}
{\sum_{j=1}^ig_j\gamma_j+\eta
g_{i+1}\gamma_{j+1}}\right)\\
&\times\left[1-\frac{\gamma_{j+1}}{2N}\left(\frac{g_{i+1}-g_i}
{\sum_{j=1}^i g_j\gamma_j+\eta
g_{i+1}\gamma_{j+1}}\right)^2+o\left(\frac{1}{N}\right)\right].
\end{aligned}
\end{equation}
Hence, the limit in \eqref{eq-h37} is
\begin{equation}\label{eq-h38}
-\frac{1}{2\gamma_{i+1}}\left(\frac{(g_{i+1}-g_i)\gamma_{i+1}}
{\sum_{j=1}^i g_j\gamma_j+\eta
g_{i+1}\gamma_{j+1}}\right)^2\exp\left(-\frac{(g_{i+1}-g_i)\gamma_{i+1}}
{\sum_{j=1}^ig_j\gamma_j+\eta g_{i+1}\gamma_{j+1}}\right).
\end{equation}
On the other hand, from \eqref{eq-h5} and \eqref{eq-h36} we find
\begin{equation}\label{eq-h39}
-\frac{(g_{i+1}-g_i)\gamma_{i+1}} {\sum_{j=1}^ig_j\gamma_j+\eta
g_{i+1}\gamma_{j+1}}=\left(\frac{g_i}{g_{i+1}}-1\right)
\ln\left(1+\frac{\gamma_{i+1}g_{i+1}}{\sum_{k=1}^i\gamma_kg_k}\right).
\end{equation}
Hence, it follows from \eqref{eq-h38} and \eqref{eq-h39} that the
limit in \eqref{eq-h37} is
\begin{equation*}
\begin{aligned}
&-\frac{1}{2\gamma_{i+1}}\left(\frac{g_i}{g_{i+1}}-1\right)^2
\left[\ln\left(1+\frac{\gamma_{i+1}g_{i+1}}{\sum_{k=1}^i\gamma_kg_k}\right)\right]^2\\
&\times\exp\left[\left(\frac{g_i}{g_{i+1}}-1\right)
\ln\left(1+\frac{\gamma_{i+1}g_{i+1}}{\sum_{k=1}^i\gamma_kg_k}\right)\right].
\end{aligned}
\end{equation*}
Similarly, for $j=0,1,\ldots, I-i$, we obtain the limit
\begin{equation}
\begin{aligned}
\label{eq-h40} &\lim_{N\to\infty}N\left\{\prod_{l=|\lfloor
N{\bf\Gamma}\rfloor|_{i+j}+1}^{|\lfloor
N{\bf\Gamma}\rfloor|_{i+j+1}} \frac{S_{\lfloor
N{\bf\Gamma}\rfloor+\mathbf{1}_i,l}^{(1)}} {S_{\lfloor
N{\bf\Gamma}\rfloor,l}^{(1)}}-\exp\left[\left(\frac{g_i}{g_{i+1+j}}-1\right)
\ln\left(1+\frac{\gamma_{i+1+j}g_{i+1+j}}{\sum_{k=1}^{i+j}\gamma_kg_k}\right)\right]
\right\}\\
=&-\frac{1}{2\gamma_{i+1+j}}\left(\frac{g_i}{g_{i+1+j}}-1\right)^2
\left[\ln\left(1+\frac{\gamma_{i+1+j}g_{i+1+j}}{\sum_{k=1}^{i+j}\gamma_kg_k}\right)
\right]^2\\
&\times\exp\left[\left(\frac{g_i}{g_{i+1+j}}-1\right)
\ln\left(1+\frac{\gamma_{i+1}g_{i+1}}{\sum_{k=1}^{i+j}\gamma_kg_k}\right)\right].
\end{aligned}
\end{equation}
Then, the limit in \eqref{eq-h40} enables us to obtain the estimate
for $\prod_{l=1}^{|\lfloor N{\bf\Gamma}\rfloor|}\frac{S_{\lfloor
N{\bf\Gamma}\rfloor+\mathbf{1}_i,l}^{(1)}} {S_{\lfloor
N{\bf\Gamma}\rfloor,l}^{(1)}}$ as $N\to\infty$:
\begin{equation}\label{eq-h41}
\begin{aligned}
\prod_{l=1}^{|\lfloor N{\bf\Gamma}\rfloor|}\frac{S_{\lfloor
N{\bf\Gamma}\rfloor+\mathbf{1}_i,l}^{(1)}} {S_{\lfloor
N{\bf\Gamma}\rfloor,l}^{(1)}}=&\theta_i^{(1)}\left\{1-\frac{1}{N}\sum_{j=1}^{I-i}
\frac{1}{2\gamma_{i+j}} \left(\frac{g_i}{g_{i+j}}-1\right)^2
\left[\ln\left(1+\frac{\gamma_{i+j}g_{i+j}}{\sum_{k=1}^{i+j}\gamma_kg_k}\right)
\right]^2\right.\\
&\times\left.\exp\left[\left(\frac{g_i}{g_{i+j}}-1\right)
\ln\left(1+\frac{\gamma_{i+1}g_{i+1}}{\sum_{k=1}^{i+j}\gamma_kg_k}\right)\right]+
o\left(\frac{1}{N}\right)\right\},
\end{aligned}
\end{equation}
from which we arrive at relation \eqref{eq-h27} for
$i=1,2,\ldots,I-1$. The proof of \eqref{eq-h30} for $i=2,3,\ldots,I$
is similar.
\end{proof}

\begin{lem}\label{lem-4}
Under the assumptions of Lemma \ref{lem-3}, we have the asymptotic
expansions:
\begin{equation}\label{eq-l21}
\begin{aligned}
P_{\lfloor N{\bf\Gamma}\rfloor}^{(1)}=&C^{(1)}(2\pi
N)^{-\frac{1}{2}(I-1)}\sqrt{\frac{1}
{\gamma_1\gamma_2\cdot\ldots\cdot\gamma_I}}
\prod_{i=1}^I\exp\left({-\alpha_i^{(1)}\gamma_i+1-\gamma_i}\right)
\left(\Delta_i^{(1)}\right)^{\lfloor N\gamma_i\rfloor}\\
&\times[1+o(1)],
\end{aligned}
\end{equation}
and
\begin{equation}\label{eq-l22}
\begin{aligned}
P_{\lfloor N{\bf\Gamma}\rfloor}^{(2)}=&C^{(2)}(2\pi
N)^{-\frac{1}{2}(I-1)}\sqrt{\frac{1}
{\gamma_1\gamma_2\cdot\ldots\cdot\gamma_I}}
\prod_{i=1}^I\exp\left({-\alpha_i^{(2)}\gamma_i+1-\gamma_i}\right)
\left(\Delta_i^{(2)}\right)^{\lfloor N\gamma_i\rfloor}\\
&\times[1+o(1)],
\end{aligned}
\end{equation}
where $\alpha_i^{(1)}$ and $\alpha_i^{(2)}$, $i=1,2,\ldots,I$, are
defined by \eqref{eq-h17}, \eqref{eq-h18}, \eqref{eq-h19} and
\eqref{eq-h20}.
\end{lem}
\begin{proof}

The proofs of \eqref{eq-l21} and \eqref{eq-l22} are similar.
Therefore, we prove \eqref{eq-l21} only. Notice that for $
\prod_{l=1}^{|\mathbf{n}|}S_{\mathbf{n},l}^{(1)} $ we have the
following obvious inequalities:
\begin{equation}\label{eq-i1}
g_1^{|\mathbf{n}|}|\mathbf{n}|!\leq\prod_{l=1}^{|\mathbf{n}|}S_{\mathbf{n},l}^{(1)}
\leq g_I^{|\mathbf{n}|}|\mathbf{n}|!.
\end{equation}
Hence, keeping in mind that for any $l$, $1\leq l<|\mathbf{n}|$, we
have
$$
\frac{S^{(1)}_{\mathbf{n},l}}{l}\leq\frac{S^{(1)}_{\mathbf{n},l+1}}{l+1},
$$
then one can arrive at the conclusion that, as $N\to\infty$, there
exists the limit
\begin{equation}\label{eq-h42}
\lim_{N\to\infty}\sqrt[N]{\frac{\prod_{l=1}^{|\lfloor
N{\bf\Gamma}\rfloor|} S_{\lfloor N{\bf\Gamma}\rfloor,l}^{(1)}}{N!}}=
\lim_{N\to\infty}\exp\left[\frac{1}{N}\sum_{l=1}^{|\lfloor
N{\bf\Gamma}\rfloor|} \ln\left(\frac{S^{(1)}_{\lfloor
N{\bf\Gamma}\rfloor,l}}{l}\right)\right]=g^{(1)}.
\end{equation}
One can find the constant $g^{(1)}$ in \eqref{eq-h42} as follows. We
have
\begin{equation}\label{eq-h16}
\lim_{N\to\infty}\frac{\prod_{l=1}^{|\lfloor N{\bf\Gamma}\rfloor|}
\frac{S^{(1)}_{\lfloor
N{\bf\Gamma}\rfloor,l}}{l}}{\prod_{l=1}^{|\lfloor
N{\bf\Gamma}\rfloor|-1}\frac{S^{(1)}_{\lfloor
N{\bf\Gamma}\rfloor,l}}{l}}=\lim_{N\to\infty}\frac{S^{(1)}_{\lfloor
N{\bf\Gamma}\rfloor,|\lfloor
N{\bf\Gamma}\rfloor|}}{N}=\lim_{N\to\infty}\frac{\sum_{i=1}^Ig_i\lfloor
N\gamma_i\rfloor}{N}=\sum_{j=1}^I \gamma_jg_j.
\end{equation}
This enables us to conclude that $g^{(1)}$ must be equal to the
right-hand side of \eqref{eq-h16}, i.e. $g^{(1)}= \sum_{i=1}^I
\gamma_ig_i$.

Let us find an asymptotic expansion for $\prod_{l=1}^{|\lfloor
N{\bf\Gamma}\rfloor|}\frac{S_{\lfloor
N{\bf\Gamma}\rfloor,l}^{(1)}}{l}$ as $N\to\infty$.

Denote  $
C_{\mathbf{n}}=\frac{|\mathbf{n}|!}{n_1!n_2!\cdot\ldots\cdot n_I!}.
$ Then,
$$
\frac{C_{\mathbf{n}+\mathbf{1}_i}}{C_{\mathbf{n}}}=\frac{{|\mathbf{n}|+1}}{n_i+1},
$$
and as $N\to\infty$, we obtain the expansion
\begin{equation}\label{eq-h14}
\frac{C_{\lfloor N{\bf\Gamma}\rfloor+\mathbf{1}_i}}{C_{\lfloor
N{\bf\Gamma}\rfloor}}=\frac{1}{\gamma_i}
\left[1-\frac{1}{N}\cdot\frac{1-\gamma_i}{\gamma_i}+O\left(\frac{1}{N^2}\right)\right].
\end{equation}
Hence, taking into account Lemma \ref{lem-3'}, we arrive at the
conclusion that an asymptotic expansion for $\prod_{l=1}^{|\lfloor
N{\bf\Gamma}\rfloor|}\frac{S_{\lfloor
N{\bf\Gamma}\rfloor,l}^{(1)}}{l}$, as $N\to\infty$, is
$$
\prod_{l=1}^{|\lfloor N{\bf\Gamma}\rfloor|}\frac{S_{\lfloor
N{\bf\Gamma}\rfloor,l}^{(1)}}{l}=\left(\sum_{i=1}^I\gamma_ig_i\right)^N
\exp\left[-\sum_{i=1}^I\left(\alpha_i^{(1)}\gamma_i-(1-\gamma_i)\right)\right]
[1+o(1)],
$$
where $\alpha_{i}^{(1)}$, $i=1,2,\ldots,I$, are given by
\eqref{eq-h17} and \eqref{eq-h18}. Now, using Stirling's formula for
\eqref{eq-d8} as $N\to\infty$, one can write the expansion
\begin{equation*}\label{eq-h44}
\begin{aligned}
P_{\lfloor N{\bf\Gamma}\rfloor}^{(1)}=&C^{(1)}(2\pi
N)^{-\frac{1}{2}(I-1)}\sqrt{\frac{1}
{\gamma_1\gamma_2\cdot\ldots\cdot\gamma_I}}
\prod_{i=1}^I\exp\left({-\alpha_i^{(1)}\gamma_i+1-\gamma_i}\right)
\left(\Delta_i^{(1)}\right)^{\lfloor N\gamma_i\rfloor}\\
&\times[1+o(1)],
\end{aligned}
\end{equation*}
and the statement of the lemma follows.
\end{proof}

\begin{lem}\label{lem-5}
There exists a positive integer $n$ such that for any vector
$\mathbf{n}\in\mathcal{N}_{{\bf\Gamma},n}$ and $i=1,2,\ldots,I$, we
have the inequalities
\begin{eqnarray}
\mathcal{J}_{i,\mathbf{n} +\mathbf{1}_i}(P_{\mathbf{n}
+\mathbf{1}_i}^{(1)},P_{\mathbf{n}
}^{(1)})-\mathcal{J}_{i,\mathbf{n}}(P_{\mathbf{n}}^{(1)},P_{\mathbf{n}
-\mathbf{1}_i}^{(1)})&\geq&0,\label{eq-l1}\\
\mathcal{J}_{i,\mathbf{n} +\mathbf{1}_i}(P_{\mathbf{n}
+\mathbf{1}_i}^{(2)},P_{\mathbf{n}
}^{(2)})-\mathcal{J}_{i,\mathbf{n}}(P_{\mathbf{n}}^{(2)},P_{\mathbf{n}
-\mathbf{1}_i}^{(2)})&\leq&0,\label{eq-l6}
\end{eqnarray}
In the case $i\neq I$, the inequalities in \eqref{eq-l1} are strong
and, respectively, in the case $i\neq 1$ the inequalities in
\eqref{eq-l6} are strong.
\end{lem}

\begin{proof} Note first, that in the case $i=I$
relation \eqref{eq-l1} and, respectively, in the case $i=1$ relation
\eqref{eq-l6} in Lemma \ref{lem-5} are automatically satisfied for
all $\mathbf{n}\in\mathcal{N}$, since in the case $i=I$ according to
relation \eqref{eq-l2} in Lemma \ref{lem-2} we have
$$\mathcal{J}_{I,\mathbf{n}+\mathbf{1}_I}(P^{(1)}_{\mathbf{n}+\mathbf{1}_I},
P^{(1)}_{\mathbf{n}})=
\mathcal{J}_{I,\mathbf{n}}(P^{(1)}_{\mathbf{n}},
P^{(1)}_{\mathbf{n}+\mathbf{1}_I})=0,$$ and in the case $i=1$
according to relation \eqref{eq-l4} in the same lemma we have
$$\mathcal{J}_{1,\mathbf{n}
+\mathbf{1}_1}(P^{(2)}_{\mathbf{n}+\mathbf{1}_1},P^{(2)}_{\mathbf{n}})=
\mathcal{J}_{1,\mathbf{n}}(P^{(2)}_{\mathbf{n}},P^{(2)}_{\mathbf{n}-\mathbf{1}_1})=0.$$
\end{proof}

 We prove now that there exists a positive integer $n$,
$\mathbf{n}\in\mathcal{N}_{{\bf\Gamma},n}$ such that the strong
inequalities of \eqref{eq-l1} hold for $i=1,2,\ldots,I-1$. The proof
of the strong inequalities of \eqref{eq-l6} for $i=2,3,\ldots,I$ is
similar.

Indeed, after canceling the multiplier
$C^{(1)}\frac{\lambda_i\prod_{l=1}^I\rho_l^{n_l}}{n_1!n_2!\cdot\ldots\cdot
n_I!}$ in the relations for
$$\mathcal{J}_{i,\mathbf{n}+\mathbf{1}_i}(P_{\mathbf{n}
+\mathbf{1}_i}^{(1)},P_{\mathbf{n}}^{(1)})-\mathcal{J}_{i,\mathbf{n}}
(P_{\mathbf{n}}^{(1)},P_{\mathbf{n}-\mathbf{1}_i}^{(1)}),
$$
and a small algebra the problem reduces to prove that there exists a
positive integer $n$ such that the inequality
\begin{equation}\label{eq-e9}
\frac{\rho_i}{n_ig_i}\left[\prod_{l=1}^{|\mathbf{n}|}S^{(1)}_{\mathbf{n}+\mathbf{1}_i,l}
-\prod_{l=1}^{|\mathbf{n}|}S^{(1)}_{\mathbf{n},l}\right]
-\prod_{l=1}^{|\mathbf{n}|-1}S^{(1)}_{\mathbf{n},l}+
\prod_{l=1}^{|\mathbf{n}|-1}S^{(1)}_{\mathbf{n}-\mathbf{1}_i,l}>0
\end{equation}
is true for all $\mathbf{n}\in\mathcal{N}_{{\bf\Gamma},n}$. Denote
the left-hand side of \eqref{eq-e9} by $f(\mathbf{n},\rho_i)$. It
follows from Lemma \ref{lem-2} that $f(\mathbf{n},0)>0$ for all
$\mathbf{n}\in\mathcal{N}$. From same Lemma \ref{lem-2} the
derivative of $f(\mathbf{n},\rho_i)$ satisfies the property
$\frac{\mathrm{d}f(\mathbf{n},\rho_i)}{\mathrm{d}\rho_i}<0$. Hence,
the lemma will be proved if we show that there exists a positive
integer $n$, such that for
$\mathbf{n}\in\mathcal{N}_{{\bf\Gamma},n}$ the function
$f(\mathbf{n},\rho_i)$ is positive for all $\rho_i$, under which the
probability mass function $P_{\mathbf{n}}^{(1)}$ is proper. From
\eqref{eq-e9} we have
\begin{equation}\label{eq-e1}
f(\rho_i)\prod_{l=1}^{|\mathbf{n}|}\frac{1}{S_{\mathbf{n},l}^{(1)}}=
\frac{\rho_i}{n_ig_i}\left[\prod_{l=1}^{|\mathbf{n}|}\frac{
S^{(1)}_{\mathbf{n}+\mathbf{1}_i,l}}{S_{\mathbf{n},l}^{(1)}}
-1\right]
-\frac{1}{\big<\mathbf{n},\mathbf{g}\big>}\left[1-\prod_{l=1}^{|\mathbf{n}|-1}
\frac{S^{(1)}_{\mathbf{n}-\mathbf{1}_i,l}}{S^{(1)}_{\mathbf{n},l}}\right].
\end{equation}
Hence, as $N\to\infty$, according to Lemma \ref{lem-3} from
\eqref{eq-e1} we obtain
\begin{equation}\label{eq-e2}
\lim_{N\to\infty}Nf(\lfloor
N{\bf\Gamma}\rfloor,\rho_i)\prod_{l=1}^{|\lfloor
N{\bf\Gamma}\rfloor|}\frac{1} {S_{\lfloor
N{\bf\Gamma}\rfloor,l}^{(1)}}
=\frac{\rho_i}{\gamma_ig_i}\left(\theta_i^{(1)}-1\right)
-\frac{1}{\sum_{l=1}^I\gamma_lg_l}\left(1-\frac{1}{\theta_i^{(1)}}\right).
\end{equation}
The right-hand side of \eqref{eq-e2} is positive, since
\begin{equation}\label{eq-e3}
\frac{\gamma_1g_1+\gamma_2g_2+\ldots+\gamma_Ig_I}{\gamma_ig_i}\cdot\theta_i^{(1)}\cdot
\rho_i=\Delta_i^{(1)}<1.
\end{equation}
Hence, for any $\rho_i$ satisfying \eqref{eq-e3}, there exists a
large value $n$ for which $f(\mathbf{n},\rho_i)>0$ for any
$\mathbf{n}\in\mathcal{N}_{{\bf\Gamma},n}$. The lemma is proved.

\subsubsection{Final part of the proof of Theorem \ref{thm-4}.}

Let us define the set $\mathcal{G}$ as the set of all directions
$\Gamma$ for which the condition $\Delta_i^{(2)}<1$,
$i=1,2,\ldots,I$, is satisfied. According to Lemma \ref{lem-5},
there exists a set of positive integer numbers $n_{\bf\Gamma}$
denoted $\mathcal{N}(\mathcal{G})$, and we define the set
$\mathcal{C}(\mathcal{G},\mathcal{N}(\mathcal{G}))
=\cup_{{\bf\Gamma}\in\mathcal{G}}\mathcal{N}_{{\bf\Gamma},n_{\bf\Gamma}}$.
Note, that the set of positive integer numbers $n_{\bf\Gamma}$ can
be chosen such that $n_{\bf\Gamma}\geq N_0$, where $N_0$ is a
sufficiently large integer number.

For all
$\mathbf{n}\in\mathcal{C}(\mathcal{G},\mathcal{N}(\mathcal{G}))$
from Lemma \ref{lem-5} we have:
\begin{eqnarray}
\sum_{i=1}^I\left[\mathcal{J}_{i,\mathbf{n}+\mathbf{1}_i}
(P_{\mathbf{n}+\mathbf{1}_i}^{(1)},
P_{\mathbf{n}}^{(1)})-\mathcal{J}_{i,\mathbf{n}}(P_{\mathbf{n}}^{(1)},
P_{\mathbf{n}-\mathbf{1}_i}^{(1)})\right]&>&0,\label{eq-h7}\\
\sum_{i=1}^I\left[\mathcal{J}_{i,\mathbf{n}+\mathbf{1}_i}
(P_{\mathbf{n}+\mathbf{1}_i}^{(2)},
P_{\mathbf{n}}^{(2)})-\mathcal{J}_{i,\mathbf{n}}(P_{\mathbf{n}}^{(2)},
P_{\mathbf{n}-\mathbf{1}_i}^{(2)})\right]&<&0.\label{eq-h8}
\end{eqnarray}
Hence, taking into account \eqref{eq-c4} together with \eqref{eq-h7}
and \eqref{eq-h8}, we can conclude that there exists the sequence of
constants $\beta_{\mathbf{n}}$, $0<\beta_{\mathbf{n}}<1$,
$\mathbf{n}\in\mathcal{C}(\mathcal{G},\mathcal{N}(\mathcal{G}))$,
such that
\begin{equation}\label{eq-h12}
\begin{aligned}
&\beta_{\mathbf{n}}\sum_{i=1}^I\left[\mathcal{J}_{i,\mathbf{n}+\mathbf{1}_i}
(P_{\mathbf{n}+\mathbf{1}_i}^{(1)}, P_{\mathbf{n}}^{(1)})-
\mathcal{J}_{i,\mathbf{n}}(P_{\mathbf{n}}^{(1)},
P_{\mathbf{n}-\mathbf{1}_i}^{(1)})\right]\\
&+(1-\beta_{\mathbf{n}})\sum_{i=1}^I\left[\mathcal{J}_{i,\mathbf{n}+\mathbf{1}_i}
(P_{\mathbf{n}+\mathbf{1}_i}^{(2)}, P_{\mathbf{n}}^{(2)})-
\mathcal{J}_{i,\mathbf{n}}(P_{\mathbf{n}}^{(2)},
P_{\mathbf{n}-\mathbf{1}_i}^{(2)})\right]\\
=&\sum_{i=1}^I\left[\mathcal{J}_{i,\mathbf{n}+\mathbf{1}_i}
(P_{\mathbf{n}+\mathbf{1}_i},P_{\mathbf{n}})-
\mathcal{J}_{i,\mathbf{n}}(P_{\mathbf{n}},P_{\mathbf{n}-\mathbf{1}_i})\right]=0.
\end{aligned}
\end{equation}

The system of equations \eqref{eq-h12} is basic for our following
study.
Notice, that for the left-hand side of \eqref{eq-h12} we have
\begin{equation*}
\begin{aligned}
&\beta_{\mathbf{n}}\sum_{i=1}^I\left[\mathcal{J}_{i,\mathbf{n}+\mathbf{1}_i}
(P_{\mathbf{n}+\mathbf{1}_i}^{(1)}, P_{\mathbf{n}}^{(1)})-
\mathcal{J}_{i,\mathbf{n}}(P_{\mathbf{n}}^{(1)},
P_{\mathbf{n}-\mathbf{1}_i}^{(1)})\right]\\
&+(1-\beta_{\mathbf{n}})\sum_{i=1}^I\left[\mathcal{J}_{i,\mathbf{n}+\mathbf{1}_i}
(P_{\mathbf{n}+\mathbf{1}_i}^{(2)}, P_{\mathbf{n}}^{(2)})-
\mathcal{J}_{i,\mathbf{n}}(P_{\mathbf{n}}^{(2)},
P_{\mathbf{n}-\mathbf{1}_i}^{(2)})\right]\\
=&\sum_{i=1}^I\left[\mathcal{J}_{i,\mathbf{n}+\mathbf{1}_i}
\big(\beta_{\mathbf{n}}P_{\mathbf{n}+\mathbf{1}_i}^{(1)}
+(1-\beta_{\mathbf{n}})P_{\mathbf{n}+\mathbf{1}_i}^{(2)},
\beta_{\mathbf{n}}P_{\mathbf{n}}^{(1)}
+(1-\beta_{\mathbf{n}})P_{\mathbf{n}}^{(2)}\big)\right.\\
&-\left.
\mathcal{J}_{i,\mathbf{n}}\big(\beta_{\mathbf{n}}P_{\mathbf{n}}^{(1)}
+(1-\beta_{\mathbf{n}})P_{\mathbf{n}}^{(2)},
\beta_{\mathbf{n}}P_{\mathbf{n}-\mathbf{1}_i}^{(1)}
+(1-\beta_{\mathbf{n}})P_{\mathbf{n}-\mathbf{1}_i}^{(2)}\big)\right],
\end{aligned}
\end{equation*}
and hence, \eqref{eq-h12} can be rewritten
\begin{equation}\label{eq-h31}
\begin{aligned}
&\sum_{i=1}^I\left[\mathcal{J}_{i,\mathbf{n}+\mathbf{1}_i}
(P_{\mathbf{n}+\mathbf{1}_i},P_{\mathbf{n}})-
\mathcal{J}_{i,\mathbf{n}}(P_{\mathbf{n}},P_{\mathbf{n}-\mathbf{1}_i})\right]\\
=&\sum_{i=1}^I\left[\mathcal{J}_{i,\mathbf{n}+\mathbf{1}_i}
\big(\beta_{\mathbf{n}}P_{\mathbf{n}+\mathbf{1}_i}^{(1)}
+(1-\beta_{\mathbf{n}})P_{\mathbf{n}+\mathbf{1}_i}^{(2)},
\beta_{\mathbf{n}}P_{\mathbf{n}}^{(1)}
+(1-\beta_{\mathbf{n}})P_{\mathbf{n}}^{(2)}\big)\right.\\
&-\left.
\mathcal{J}_{i,\mathbf{n}}\big(\beta_{\mathbf{n}}P_{\mathbf{n}}^{(1)}
+(1-\beta_{\mathbf{n}})P_{\mathbf{n}}^{(2)},
\beta_{\mathbf{n}}P_{\mathbf{n}-\mathbf{1}_i}^{(1)}
+(1-\beta_{\mathbf{n}})P_{\mathbf{n}-\mathbf{1}_i}^{(2)}\big)\right].
\end{aligned}
\end{equation}

The left-hand side of \eqref{eq-h31} equated to zero defines the
system of linear equations for $P_{\mathbf{n}}$, and the right-hand
of \eqref{eq-h31} equated to zero defines the system of equations
for the convex combination
$\beta_{\mathbf{n}}P_{\mathbf{n}}^{(1)}+(1-\beta_{\mathbf{n}})P_{\mathbf{n}}^{(2)}$.
For $\mathbf{n}\in\mathcal{C}(\mathcal{G},\mathcal{N}(\mathcal{G}))$
these systems of equations are identical. However, they are
expressed via the values $P_{\mathbf{n}-\mathbf{1}_i}$ and
$\beta_{\mathbf{n}}P_{\mathbf{n}-\mathbf{1}_i}^{(1)}+(1-
\beta_{\mathbf{n}})P_{\mathbf{n}-\mathbf{1}_i}^{(2)}$, respectively,
in the first and second equations, in which if $\mathbf{n}$ is a
boundary vector of
$\mathcal{C}(\mathcal{G},\mathcal{N}(\mathcal{G}))$, the vector
$\mathbf{n}-\mathbf{1}_i$ may not belong to the set
$\mathcal{C}(\mathcal{G},\mathcal{N}(\mathcal{G}))$.

For
$\{\mathbf{n},i\}\in\mathcal{C}^0(\mathcal{G},\mathcal{N}(\mathcal{G}))$,
let $P_{\mathbf{n}-\mathbf{1}_i}=d_{\mathbf{n},i}$
and let
\begin{equation*}\label{eq-e6}
\sum_{\mathbf{n}\in\mathcal{C}(\mathcal{G},\mathcal{N}(\mathcal{G}))}P_{\mathbf{n}}=p>0,
\end{equation*}
and
\begin{equation*}\label{eq-e7}
\sum_{\mathbf{n}\in\mathcal{C}(\mathcal{G},\mathcal{N}(\mathcal{G}))}
[\beta_{\mathbf{n}}P_{\mathbf{n}}^{(1)}
+(1-\beta_{\mathbf{n}})P_{\mathbf{n}}^{(2)}]=p^*>0,
\end{equation*}
where the constants $p$ and $p^*$ are the normalizing constants, $p$
depends on the values $d_{\mathbf{n},i}$
. Notice, that for the left-hand side of \eqref{eq-h31} equated to
zero,
$\mathbf{n}\in\mathcal{C}(\mathcal{G},\mathcal{N}(\mathcal{G}))$, we
have the following system of equations:
\begin{equation}\label{eq-h23}
\sum_{i=1}^I\left[\mathcal{J}_{i,\mathbf{n}+\mathbf{1}_i}
\left(\frac{P_{\mathbf{n}+\mathbf{1}_i}}{P_{\mathbf{n}}},1\right)-
\mathcal{J}_{i,\mathbf{n}}\left(1,\frac{P_{\mathbf{n}-\mathbf{1}_i}}{P_{\mathbf{n}}}
\right)\right]=0.
\end{equation}
For the right-hand side of \eqref{eq-h31} equated to zero,
$\mathbf{n}\in\mathcal{C}(\mathcal{G},\mathcal{N}(\mathcal{G}))$,
the system of equations is similar:
\begin{equation}\label{eq-h24}
\begin{aligned}
&\sum_{i=1}^I\left[\mathcal{J}_{i,\mathbf{n}+\mathbf{1}_i}
\left(\frac{\beta_{\mathbf{n}}P_{\mathbf{n}+\mathbf{1}_i}^{(1)}
+(1-\beta_{\mathbf{n}})P_{\mathbf{n}+\mathbf{1}_i}^{(2)}}{
\beta_{\mathbf{n}}P_{\mathbf{n}}^{(1)}
+(1-\beta_{\mathbf{n}})P_{\mathbf{n}}^{(2)}},1\right)\right.\\
&-\left. \mathcal{J}_{i,\mathbf{n}}\left(1,
\frac{\beta_{\mathbf{n}}P_{\mathbf{n}-\mathbf{1}_i}^{(1)}
+(1-\beta_{\mathbf{n}})P_{\mathbf{n}-\mathbf{1}_i}^{(2)}}
{\beta_{\mathbf{n}}P_{\mathbf{n}}^{(1)}
+(1-\beta_{\mathbf{n}})P_{\mathbf{n}}^{(2)}}\right)\right]=0.
\end{aligned}
\end{equation}

With the aforementioned initial conditions
$P_{\mathbf{n}-\mathbf{1}_i}=d_{\mathbf{n},i}$,
$\{\mathbf{n},i\}\in\mathcal{C}^0(\mathcal{G},\mathcal{N}(\mathcal{G}))$,
the system of equations \eqref{eq-h23}
is uniquely defined. For any $\epsilon>0$ the value $N_0$ can be
chosen so large, that for $N\geq N_0$ and all $i=1,2,\ldots,I$,
\begin{equation}\label{eq-h45}
\left|\frac{P_{\lfloor N{\bf\Gamma}\rfloor-\mathbf{1}_i}}{P_{\lfloor
N{\bf\Gamma}\rfloor}}-\frac{P_{\lfloor
N{\bf\Gamma}\rfloor}}{P_{\lfloor
N{\bf\Gamma}\rfloor+\mathbf{1}_i}}\right|<\epsilon,
\end{equation}
and similarly we have
\begin{equation}\label{eq-h46}
\begin{aligned}
&\left|\frac {\beta_{\lfloor N{\bf\Gamma}\rfloor}P_{\lfloor
N{\bf\Gamma}\rfloor-\mathbf{1}_i}^{(1)} +(1-\beta_{\lfloor
N{\bf\Gamma}\rfloor})P_{\lfloor
N{\bf\Gamma}\rfloor-\mathbf{1}_i}^{(2)}} {\beta_{\lfloor
N{\bf\Gamma}\rfloor}P_{\lfloor N{\bf\Gamma}\rfloor}^{(1)}
+(1-\beta_{\lfloor N{\bf\Gamma}\rfloor})P_{\lfloor
N{\bf\Gamma}\rfloor}^{(2)}}\right.\\
&-\left.\frac {\beta_{\lfloor
N{\bf\Gamma}\rfloor+\mathbf{1}_i}P_{\lfloor
N{\bf\Gamma}\rfloor}^{(1)} +(1-\beta_{\lfloor
N{\bf\Gamma}\rfloor+\mathbf{1}_i})P_{\lfloor
N{\bf\Gamma}\rfloor}^{(2)}} {\beta_{\lfloor
N{\bf\Gamma}\rfloor+\mathbf{1}_i}P_{\lfloor
N{\bf\Gamma}\rfloor+\mathbf{1}_i}^{(1)} +(1-\beta_{\lfloor
N{\bf\Gamma}\rfloor+\mathbf{1}_i})P_{\lfloor
N{\bf\Gamma}\rfloor+\mathbf{1}_i}^{(2)}}\right|<\epsilon.
\end{aligned}
\end{equation}
On the other hand, the system of equations \eqref{eq-h31} implies a
continuous correspondence between the systems of equations
\eqref{eq-h23} and \eqref{eq-h24}. It means that for any $\iota>0$
there exists the value $\epsilon>0$ such that \eqref{eq-h45} and
\eqref{eq-h46} imply
\begin{equation*}
\left|\frac{P_{\lfloor N{\bf\Gamma}\rfloor-\mathbf{1}_i}}{P_{\lfloor
N{\bf\Gamma}\rfloor}}-\frac {\beta_{\lfloor
N{\bf\Gamma}\rfloor}P_{\lfloor
N{\bf\Gamma}\rfloor-\mathbf{1}_i}^{(1)} +(1-\beta_{\lfloor
N{\bf\Gamma}\rfloor})P_{\lfloor
N{\bf\Gamma}\rfloor-\mathbf{1}_i}^{(2)}} {\beta_{\lfloor
N{\bf\Gamma}\rfloor}P_{\lfloor N{\bf\Gamma}\rfloor}^{(1)}
+(1-\beta_{\lfloor N{\bf\Gamma}\rfloor})P_{\lfloor
N{\bf\Gamma}\rfloor}^{(2)}}\right|<\iota,
\end{equation*}
which in turn means
\begin{equation}\label{eq-h13}
\lim_{N\to\infty}\frac{P_{\lfloor
N{\bf\Gamma}\rfloor+\mathbf{1}_i}}{P_{\lfloor N{\bf\Gamma}\rfloor}}=
\lim_{N\to\infty}\frac {\beta_{\lfloor
N{\bf\Gamma}\rfloor}P_{\lfloor
N{\bf\Gamma}\rfloor+\mathbf{1}_i}^{(1)} +(1-\beta_{\lfloor
N{\bf\Gamma}\rfloor})P_{\lfloor
N{\bf\Gamma}\rfloor+\mathbf{1}_i}^{(2)}} {\beta_{\lfloor
N{\bf\Gamma}\rfloor}P_{\lfloor N{\bf\Gamma}\rfloor}^{(1)}
+(1-\beta_{\lfloor N{\bf\Gamma}\rfloor})P_{\lfloor
N{\bf\Gamma}\rfloor}^{(2)}}.
\end{equation}
Applying asymptotic relations \eqref{eq-e20} and \eqref{eq-e21} of
Lemma \ref{lem-3} to the right-hand side of \eqref{eq-h13} we obtain
\begin{equation}\label{eq-h15}
\lim_{N\to\infty}\frac{P_{\lfloor
N{\bf\Gamma}\rfloor+\mathbf{1}_i}}{P_{\lfloor N{\bf\Gamma}\rfloor}}=
\lim_{N\to\infty}\frac {\beta_{\lfloor
N{\bf\Gamma}\rfloor}\Delta_i^{(1)}P_{\lfloor
N{\bf\Gamma}\rfloor}^{(1)} +(1-\beta_{\lfloor
N{\bf\Gamma}\rfloor})\Delta_i^{(2)}P_{\lfloor
N{\bf\Gamma}\rfloor}^{(2)}} {\beta_{\lfloor
N{\bf\Gamma}\rfloor}P_{\lfloor N{\bf\Gamma}\rfloor}^{(1)}
+(1-\beta_{\lfloor N{\bf\Gamma}\rfloor})P_{\lfloor
N{\bf\Gamma}\rfloor}^{(2)}}.
\end{equation}

The two propositions below enable us to establish the limit in
\eqref{eq-h15}. In the following, for two sequences $\{a_N\}$ and
$\{b_N\}$ vanishing as $N\to\infty$, the writing $a_N\sim b_N$ means
that $\lim_{N\to\infty}\frac{a_N}{b_N}=1$.

\begin{prop}\label{prop-1} As $N\to\infty$,
\begin{equation}\label{eq-h9}
\begin{aligned}
&\sum_{i=1}^I\left[\mathcal{J}_{i,\lfloor
N{\bf\Gamma}\rfloor+\mathbf{1}_i}( P_{\lfloor
N{\bf\Gamma}\rfloor+\mathbf{1}_i}^{(1)}, P_{\lfloor
N{\bf\Gamma}\rfloor}^{(1)})- \mathcal{J}_{i,\lfloor
N{\bf\Gamma}\rfloor}(P_{\lfloor N\Gamma\rfloor}^{(1)}, P_{\lfloor
N{\bf\Gamma}\rfloor-\mathbf{1}_i}^{(1)})\right]\\ &\sim P_{\lfloor
N{\bf\Gamma}\rfloor}^{(1)}\sum_{i=1}^I\mathcal{J}_{i,\lfloor
N{\bf\Gamma}\rfloor}\left( {\Delta_i^{(1)}}-1,
1-\frac{1}{\Delta_i^{(1)}}\right),
\end{aligned}
\end{equation}
and
\begin{equation}\label{eq-h10}
\begin{aligned}
&\sum_{i=1}^I\left[\mathcal{J}_{i,\lfloor
N{\bf\Gamma}\rfloor+\mathbf{1}_i}( P_{\lfloor
N{\bf\Gamma}\rfloor+\mathbf{1}_i}^{(2)}, P_{\lfloor
N{\bf\Gamma}\rfloor}^{(2)})- \mathcal{J}_{i,\lfloor
N{\bf\Gamma}\rfloor}(P_{\lfloor N{\bf\Gamma}\rfloor}^{(2)},
P_{\lfloor
N{\bf\Gamma}\rfloor-\mathbf{1}_i}^{(2)})\right]\\
&\sim P_{\lfloor
N{\bf\Gamma}\rfloor}^{(2)}\sum_{i=1}^I\mathcal{J}_{i,\lfloor
N{\bf\Gamma}\rfloor}\left( {\Delta_i^{(2)}}-1,
1-\frac{1}{\Delta_i^{(2)}}\right).
\end{aligned}
\end{equation}
\end{prop}

\begin{proof} From Lemma \ref{lem-3}, we have
$\lim_{N\to\infty}\frac{P_{\lfloor
N{\bf\Gamma}\rfloor+\mathbf{1}_i}^{(1)}}{P_{\lfloor
N{\bf\Gamma}\rfloor}^{(1)}} =\Delta_i^{(1)}$ and
$\lim_{N\to\infty}\frac{P_{\lfloor
N{\bf\Gamma}\rfloor+\mathbf{1}_i}^{(2)}}{P_{\lfloor
N{\bf\Gamma}\rfloor}^{(2)}} =\Delta_i^{(2)}$. Assume that $N_0$ is
chosen large enough such that for all $N>N_0$ the inequalities
$|P_{\lfloor N{\bf\Gamma}\rfloor}^{(1)}\Delta_i^{(1)}-P_{\lfloor
N{\bf\Gamma}\rfloor-\mathbf{1}_i}^{(1)}| <\epsilon P_{\lfloor
N{\bf\Gamma}\rfloor-\mathbf{1}_i}^{(1)}$ and $|P_{\lfloor
N{\bf\Gamma}\rfloor}^{(2)}\Delta_i^{(2)}-P_{\lfloor
N{\bf\Gamma}\rfloor-\mathbf{1}_i}^{(2)}| <\epsilon P_{\lfloor
N{\bf\Gamma}\rfloor-\mathbf{1}_i}^{(2)}$, $i=1,2,\ldots,I$ are
satisfied. Then, for the upper bound we obtain
\begin{equation}\label{eq-h25}
\begin{aligned}
&\sum_{i=1}^I\left[\mathcal{J}_{i,\lfloor
N{\bf\Gamma}\rfloor+\mathbf{1}_i} \left( \frac{P_{\lfloor
N{\bf\Gamma}\rfloor+\mathbf{1}_i}^{(1)}}{P_{\lfloor
N{\bf\Gamma}\rfloor}^{(1)}}, 1\right)- \mathcal{J}_{i,\lfloor
N{\bf\Gamma}\rfloor}\left(1, \frac{P_{\lfloor
N{\bf\Gamma}\rfloor-\mathbf{1}_i}^{(1)}}{P_{\lfloor
N{\bf\Gamma}\rfloor}^{(1)}}\right)\right]\\
\leq&\sum_{i=1}^I\left[\mathcal{J}_{i,\lfloor
N{\bf\Gamma}\rfloor+\mathbf{1}_i}\left( {\Delta_i^{(1)}}, 1\right)-
\mathcal{J}_{i,\lfloor N{\bf\Gamma}\rfloor}\left(1,
\frac{1}{\Delta_i^{(1)}}\right)\right]\\
&+\epsilon \sum_{i=1}^I\mu_ig_i\left[\frac{\lfloor N\gamma_i\rfloor
+1} {\big<\lfloor N{\bf\Gamma}\rfloor+\mathbf{1}_i,\mathbf{g}\big>}
+\frac{\lambda_i}
{\Delta_i^{(1)}}\right]\\
=&\sum_{i=1}^I\left[\mathcal{J}_{i,\lfloor
N{\bf\Gamma}\rfloor}\left(
{\Delta_i^{(1)}}-1, 1-\frac{1}{\Delta_i^{(1)}}\right)\right]\\
&+\epsilon \sum_{i=1}^I\mu_ig_i\left[\frac{\lfloor
N\gamma_i\rfloor+1} {\big<\lfloor
N{\bf\Gamma}\rfloor+\mathbf{1}_i,\mathbf{g}\big>} +\frac{\lambda_i}
{\Delta_i^{(1)}}\right]\\
&+ \sum_{i=1}^I\mu_ig_i\left[\frac{\lfloor
N\gamma_i\rfloor+1}{\big<\lfloor N{\bf\Gamma}\rfloor
+\mathbf{1}_i,\mathbf{g}\big>}-\frac{\lfloor
N\gamma_i\rfloor}{\big<\lfloor
N{\bf\Gamma}\rfloor,\mathbf{g}\big>}\right].
\end{aligned}
\end{equation}
The estimate for the lower bound is similar.

Clearly, the terms $$\epsilon
\sum_{i=1}^I\mu_ig_i\left[\frac{\lfloor N\gamma_i\rfloor+1}
{\big<\lfloor N\Gamma\rfloor+\mathbf{1}_i,\mathbf{g}\big>}
+\frac{\lambda_i} {\Delta_i^{(1)}}\right]$$ and
$$\sum_{i=1}^I\mu_ig_i\left[\frac{\lfloor
N\gamma_i\rfloor+1}{\big<\lfloor N\Gamma\rfloor
+\mathbf{1}_i,\mathbf{g}\big>}-\frac{\lfloor
N\gamma_i\rfloor}{\big<\lfloor
N\Gamma\rfloor,\mathbf{g}\big>}\right]$$ in \eqref{eq-h25} vanish as
$\epsilon\to0$ and $N\to\infty$. Hence, \eqref{eq-h9} follows. The
proof of  \eqref{eq-h10} is similar.
\end{proof}

\begin{prop}\label{prop-2}
As $N\to\infty$, the sequence of $\beta_{\lfloor
N{\bf\Gamma}\rfloor}$ tends to $1$. Furthermore,
$$
1-\beta_{\lfloor
N{\bf\Gamma}\rfloor}=\frac{\sum_{i=1}^I\left[\mu_i{\gamma_ig_i}{\big[\sum_{l=1}^I
\gamma_lg_l\big]^{-1}}(\Delta_i^{(1)}-1)
-\lambda_i\big(1-\big(\Delta_i^{(1)}\big)^{-1}\big)\right]}{\sum_{i=1}^I\left[
\mu_i{\gamma_ig_i}
{\big[\sum_{l=1}^I\gamma_lg_l\big]^{-1}}(\Delta_i^{(2)}-1)
-\lambda_i\big(1-\big(\Delta_i^{(2)}\big)^{-1}\big)\right]}
$$
\begin{equation}\label{eq-h21}
\times\frac{C^{(1)}}{C^{(2)}}\prod_{i=1}^I\mathrm{e}^{\left(\alpha_i^{(2)}
-\alpha_i^{(1)}\right)\gamma_i}\left(\frac{\Delta_i^{(1)}}
{\Delta_i^{(2)}}\right)^{\lfloor N\gamma_i\rfloor}[1+o(1)],
\end{equation}
where $\alpha_{i}^{(1)}$ and $\alpha_i^{(2)}$ are given by
\eqref{eq-h17}, \eqref{eq-h18}, \eqref{eq-h19} and \eqref{eq-h20}.
\end{prop}
\begin{proof} It follows from Lemma \ref{lem-4} that
$P_{\lfloor
N{\bf\Gamma}\rfloor}^{(1)}=O\left(N^{-\frac{1}{2}(I-1)}\prod_{i=1}^I
\left[\Delta_i^{(1)}\right]^{\lfloor N\gamma_i\rfloor}\right)$ and
$P_{\lfloor
N{\bf\Gamma}\rfloor}^{(2)}=O\left(N^{-\frac{1}{2}(I-1)}\prod_{i=1}^I
\left[\Delta_i^{(2)}\right]^{\lfloor N\gamma_i\rfloor}\right)$ as
$N\to\infty$. Furthermore, it is readily seen from the explicit
expressions of \eqref{eq-d11} and \eqref{eq-d12} that
$\theta_i^{(1)}<1$, $i=1,2,\ldots,I-1$, $\theta_I^{(1)}=1$, and
$\theta_i^{(2)}>1$, $i=2,\ldots,I$, $\theta_1^{(2)}=1$, i.e.
$\theta_i^{(1)}<\theta_i^{(2)}$, $i=1,2,\ldots,I$. Consequently,
$\Delta_i^{(1)}<\Delta_i^{(2)}$, $i=1,2,\ldots,I.$ Recall also that
$\Delta_i^{(2)}<1$. Hence, from Proposition \ref{prop-1} we have:
\begin{equation}\label{eq-h22}
\begin{aligned}
&\lim_{N\to\infty}\prod_{l=1}^I\left(\frac{1}{\Delta_{l}^{(1)}}\right)^{\lfloor
N\gamma_l\rfloor}\sum_{i=1}^I\left[ \mathcal{J}_{i,\lfloor
N{\bf\Gamma}\rfloor+\mathbf{1}_i}( P_{\lfloor
N{\bf\Gamma}\rfloor+\mathbf{1}_i}^{(1)}, P_{\lfloor
N{\bf\Gamma}\rfloor}^{(1)})\right.\\
&\hskip 110pt\left.- \mathcal{J}_{i,\lfloor
N{\bf\Gamma}\rfloor}(P_{\lfloor N{\bf\Gamma}\rfloor}^{(1)},
P_{\lfloor N{\bf\Gamma}\rfloor-\mathbf{1}_i}^{(1)})\right]=0,
\end{aligned}
\end{equation}
and
\begin{equation}\label{eq-h26}
\begin{aligned}
&\lim_{N\to\infty}\prod_{l=1}^I\left(\frac{1}{\Delta_{l}^{(1)}}\right)^{\lfloor
N\gamma_l\rfloor}\sum_{i=1}^I\left[ \mathcal{J}_{i,\lfloor
N{\bf\Gamma}\rfloor+\mathbf{1}_i}( P_{\lfloor
N{\bf\Gamma}\rfloor+\mathbf{1}_i}^{(2)}, P_{\lfloor
N{\bf\Gamma}\rfloor}^{(2)})\right.\\
&\hskip 110pt\left.- \mathcal{J}_{i,\lfloor
N{\bf\Gamma}\rfloor}(P_{\lfloor N{\bf\Gamma}\rfloor}^{(2)},
P_{\lfloor N{\bf\Gamma}\rfloor-\mathbf{1}_i}^{(2)})\right]=\infty.
\end{aligned}
\end{equation}
It follows from \eqref{eq-h22} and \eqref{eq-h26} that
$\beta_{\lfloor N{\bf\Gamma}\rfloor}$ tends to 1 as $N\to\infty$. To
obtain the exact expansion given by \eqref{eq-h21} we take into
account \eqref{eq-l21} and \eqref{eq-l22} of Lemma \ref{lem-4} and
 \eqref{eq-h9} and \eqref{eq-h10} of Proposition
\ref{prop-1}. From these estimates we obtain
\begin{eqnarray*}
1-\beta_{\lfloor N{\bf\Gamma}\rfloor}
&\sim&-\frac{P_{\lfloor
N{\bf\Gamma}\rfloor}^{(1)}\sum_{i=1}^I\mathcal{J}_{i,\lfloor
N{\bf\Gamma}\rfloor} \left(\Delta_i^{(1)}-
1,1-\big(\Delta_i^{(1)}\big)^{-1} \right)}{P_{\lfloor
N{\bf\Gamma}\rfloor}^{(2)}\sum_{i=1}^I\mathcal{J}_{i,\lfloor
N{\bf\Gamma}\rfloor}\left(\Delta_i^{(2)}-
1,1-\big(\Delta_i^{(2)}\big)^{-1}
\right)}\nonumber\\
&\sim&-\frac{\sum_{i=1}^I\left[\mu_i{\gamma_ig_i}{\big[\sum_{l=1}^I
\gamma_lg_l\big]^{-1}}(\Delta_i^{(1)}-1)
-\lambda_i\big(1-\big(\Delta_i^{(1)}\big)^{-1}\big)\right]}{\sum_{i=1}^I
\left[\mu_i{\gamma_ig_i}
{\big[\sum_{l=1}^I\gamma_lg_l\big]^{-1}}(\Delta_i^{(2)}-1)
-\lambda_i\big(1-\big(\Delta_i^{(2)}\big)^{-1}\big)\right]}\nonumber\\
&&\times\frac{C^{(1)}}{C^{(2)}}\prod_{i=1}^I\mathrm{e}^{\left(\alpha_i^{(2)}
-\alpha_i^{(1)}\right)\gamma_i} \left(\frac{\Delta_i^{(1)}}
{\Delta_i^{(2)}}\right) ^{\lfloor N\gamma_i\rfloor}.
\end{eqnarray*}
Proposition \ref{prop-2} is proved.
\end{proof}

Let us now calculate the limit in the left-hand side of
\eqref{eq-h15}. Inserting \eqref{eq-h21} into the limit in the
right-hand side of \eqref{eq-h15}, with the aid of asymptotic
expansions \eqref{eq-l21} and \eqref{eq-l22} of Lemma \ref{lem-4} we
obtain
$$
\lim_{N\to\infty}\frac{P_{\mathbf{n}+\mathbf{1}_i}}{P_{\mathbf{n}}}=\frac
{\Delta_i^{(1)}-c\Delta_i^{(2)}}{1-c},
$$
where
$$
c=\frac{\sum_{i=1}^I\left[\mu_i{\gamma_ig_i}{\big(\sum_{l=1}^I
\gamma_lg_l\big)^{-1}}(\Delta_i^{(1)}-1)
-\lambda_i\big(1-\big(\Delta_i^{(1)}\big)^{-1}\big)\right]}{\sum_{i=1}^I
\left[\mu_i{\gamma_ig_i}
{\big(\sum_{l=1}^I\gamma_lg_l\big)^{-1}}(\Delta_i^{(2)}-1)
-\lambda_i\big(1-\big(\Delta_i^{(2)}\big)^{-1}\big)\right]}.
$$
Theorem \ref{thm-4} is proved.

\subsubsection{The proof of Corollary \ref{cor-1}.}
Assume first that $\gamma_1$, $\gamma_2$,\ldots, $\gamma_I$ are
rational numbers. Denote $\delta=\inf\{x: \gamma_ix\in\mathbb{N},
i=1,2,\ldots,I\}$. Then, from Theorem \ref{thm-4} we obtain
$$
\lim_{N\to\infty}\frac{P_{\lfloor N{\bf\Gamma}\rfloor
+\delta{\bf\Gamma}}}{P_{\lfloor
N{\bf\Gamma}\rfloor}}=\prod_{i=1}^I\left(\Delta_i\right)^{\delta\gamma_i}.
$$
Then, denoting $M=\max\{m: \delta m\leq N\}$ we obtain
\begin{equation}\label{eq-h11.1}
\begin{aligned}
\delta \sum_{i=1}^I\gamma_i\ln\Delta_i=&\lim_{N\to\infty}(\ln
P_{\lfloor N{\bf\Gamma}\rfloor +\delta{\bf\Gamma}}-\ln P_{\lfloor
N{\bf\Gamma}\rfloor})\\
=&\lim_{N\to\infty}\frac{\delta}{N}\left[\sum_{m=1}^{M-1}(\ln
P_{\delta
(m+1){\bf\Gamma}}-\ln P_{\delta m{\bf\Gamma}})\right]\\
&+\underbrace{\lim_{N\to\infty}\frac{\delta}{N}\left[\ln P_{\lfloor
N{\bf\Gamma}\rfloor}-\ln
P_{\delta M{\bf\Gamma}}\right]}_{=0}\\
&+\underbrace{\lim_{N\to\infty}\frac{\delta}{N}\left[\ln P_{\lfloor
N{\bf\Gamma}\rfloor+\delta{\bf\Gamma}}-\ln P_{\lfloor
N{\bf\Gamma}\rfloor}\right]}_{=0}\\
=&\lim_{N\to\infty}\frac{\delta}{N}\ln P_{\delta M{\bf\Gamma}}=
\lim_{N\to\infty}\frac{\delta}{N}\ln P_{\lfloor
N{\bf\Gamma}\rfloor}.
\end{aligned}
\end{equation}
Hence,
\begin{equation}\label{eq-h11}
\lim_{N\to\infty}\frac{\ln P_{\lfloor
N{\bf\Gamma}\rfloor}}{N}=\sum_{i=1}^I\gamma_i\ln\Delta_i.
\end{equation}
Assume now, that there is at least one of $\gamma_1$,
$\gamma_2$,\ldots, $\gamma_I$ that is irrational. Then, there is a
sequence of rational numbers $\gamma_{1,n}$, $\gamma_{2,n}$,\ldots,
$\gamma_{I,n}$ that converges to the limit $\gamma_1$,
$\gamma_2$,\ldots, $\gamma_I$. Denote $\delta_n=\inf\{x:
\gamma_{i,n}x\in\mathbb{N}, i=1,2,\ldots,I\}$. Then, for any
$\{\gamma_{1,n},\gamma_{2,n},\ldots, \gamma_{I,n}\}$, the limiting
relation of \eqref{eq-h11} holds. Then, keeping in mind that
$\Delta_i=\frac{\Delta_i^{(1)}-c\Delta_i^{(2)}}{1-c}$ is continuous
in $\gamma_1$, $\gamma_2$,\ldots, $\gamma_I$, because each of
$\Delta_i^{(1)}$, $\Delta_i^{(2)}$ and $c$ is continuous in
$\gamma_1$, $\gamma_2$,\ldots, $\gamma_i$, one can take a limit in
\eqref{eq-h11.1} as $\delta_n$ increases to infinity to arrive at
\eqref{eq-h11}. The corollary is proved.

\section{Most likely asymptotic direction}\label{Num}

Asymptotic Theorem \ref{thm-4} is obtained under the assumption that
$n_1=\lfloor N\gamma_1\rfloor$, $n_2=\lfloor N\gamma_2\rfloor$,
\ldots, $n_I=\lfloor N\gamma_I\rfloor$ for large value $N$. By
\textit{most likely direction} we mean such values $\gamma_1$,
$\gamma_2$, \ldots, $\gamma_I$ that minimize $
\big(-\lim_{N\to\infty}\frac{1}{N}\ln P_{\lfloor
N{\bf\Gamma}\rfloor}\big).$
 Then, the problem is
to minimize
$$-\sum_{i=1}^I\gamma_i\ln\frac{\Delta_i^{(1)}-c\Delta_i^{(2)}}{1-c},$$
subject to the constraints
\begin{eqnarray*}
\Delta_i^{(2)}&<&1,\quad i=1,2,\ldots, I,\\
\sum_{l=1}^I\gamma_lg_l&<&\gamma_ig_i\Delta_i^{(2)},\quad i=1,2,\ldots, I,\\
\sum_{i=1}^I\gamma_i&=&1.
\end{eqnarray*}
This is a convex optimization problem. It can be solved by the
interior point method \cite{Boyd}. Some numerical examples for its
solution are given in Table 1. For the numerical calculations, the
following set of parameters is taken: $I=2$, $\mu_1=\mu_2=1$ and
$\lambda_1=0.2$, $\lambda_2=0.3$. The value $g_1=2$ is taken fixed
for all calculations in the table. The variable parameter $g_2$
takes the values 2, 2.5, 3, 3.5 and 4. The case where $g_2=g_1=2$
(the first row in the table) is related to the Egalitarian PS
system.

\begin{table}
    \begin{center}
        \caption{The table of optimal values of $\gamma_1$ and $\gamma_2$ for DPS systems
        with two classes of flow}
            \begin{tabular}{c|c|c}\hline
            Variable parameter & Optimal value & Optimal value\\
            $g_2$              & $\gamma_1$ & $\gamma_2$\\
            \hline
            2                  & .401       & .599\\
            2.5                & .435       & .565\\
            3                  & .474       & .526\\
            3.5                & .633       & .367\\
            4                  & .622       & .378\\
            \hline
            \end{tabular}
    \end{center}
\end{table}

\section{Concluding remarks and an open problem}\label{Conclusion}

In the present paper we established a characterization theorem on
impossibility of presenting the stationary probabilities in closed
geometric form. Implicitly we have shown that the stationary
probability cannot have the  form
$F(\mathbf{n},\mathbf{g})G(\mathbf{n},\lambda,\mu)$, where
$\lambda=\{\lambda_1,\lambda_2,\ldots,\lambda_I\}$ and
$\mu=\{\mu_1,\mu_2,\ldots,\mu_I\}$, since if $P_{\mathbf{n}}$ can be
represented in this form, then it can be shown that
$G(\mathbf{n},\lambda,\mu)$ must be equal to
$(1-\rho)\prod_{i=1}^I\rho_i^{n_i}$.

While for Egalitarian PS systems the explicit formula for the
stationary distribution is known and has a relatively simple closed
geometric form, the analysis of the DPS system it very hard. We have
provided a full asymptotic analysis of the tail probabilities that
is based on an analysis of the system of the equations for the
stationary probabilities. The method of asymptotic analysis uses
technical assumption \eqref{eq-l7} that includes the constants
$\theta_i^{(2)}>1$, $i=2,3,\ldots, I$ and $\theta_1^{(2)}=1$.

Unfortunately, the methods of asymptotic analysis of the present
paper enables us to merely obtain the asymptotes for
$\frac{P_{\lfloor N{\bf\Gamma}\rfloor+\mathbf{1}_i}}{P_{\lfloor
N{\bf\Gamma}\rfloor}}$, $i=1,2,\ldots,I$, for large $N$, but do not
permit to obtain an asymptotic expansion for $P_{\lfloor
N{\bf\Gamma}\rfloor}$ itself. This type of asymptotic expansion
requires more delicate methods of asymptotic analysis. Our
conjecture is that under the assumptions made in Theorem
\ref{thm-4}, $$P_{\lfloor
N{\bf\Gamma}\rfloor}=O\left(N^{-\frac{1}{2}(I-1)}\prod_{i=1}^I
\Big[\frac{\Delta_i^{(1)}-c\Delta_i^{(2)}} {1-c}\Big]^{\lfloor
N\gamma_i\rfloor}\right).$$ The probabilities $P_{\lfloor
N{\bf\Gamma}\rfloor}^{(1)}$ and $P_{\lfloor
N{\bf\Gamma}\rfloor}^{(2)}$ have the similar type of asymptotes
(Lemma \ref{lem-4}), and this is the reason for this conjecture.

\section{Acknowledgements} The author thanks Lachlan Andrew and Yoni Nazarathy for
useful discussions of the problem and help in preparation of the
paper. The author thanks Matthieu Jonckheere for bringing to the
attention of the author the paper by Bonald and Prouti\`{e}re
\cite{BonaldProutiere3} and Nadezhda Suhorukova for providing the
author with materials related to the interior point method. The work
was supported by the Australian Research Council (ARC), grant
DP0985322.


%
%
%
%

\end{document}